\documentclass[11pt,a4paper]{amsart}
\usepackage{amsfonts}
\usepackage{amsthm}
\usepackage{amsmath}
\usepackage{amscd}
\usepackage[latin2]{inputenc}
\usepackage{t1enc}
\usepackage[mathscr]{eucal}
\usepackage{indentfirst}
\usepackage{graphicx}
\usepackage{graphics}
\usepackage{pict2e}
\usepackage{epic}
\numberwithin{equation}{section}
\usepackage[margin=2.7cm]{geometry}
\usepackage{epstopdf} 
\usepackage[table]{xcolor}
\usepackage{longtable}
\usepackage{changepage}
\usepackage{microtype}
\usepackage{float}
\usepackage[noadjust]{cite}

\theoremstyle{plain}
\newtheorem{theorem}{Theorem}[section]

\newtheorem{proposition}[theorem]{Proposition}

\newtheorem{lemma}[theorem]{Lemma}
\theoremstyle{definition}
\newtheorem{definition}{Definition}[section]

\newtheorem{remark} [theorem] {Remark}

\newcommand{\C}{\mathbb{C}}
\newcommand{\Z}{\mathbb{Z}}
\newcommand{\Q}{\mathbb{Q}}

\newcommand{\F}{\mathbb{F}}

\newcommand{\PP}{\mathbb{P}}
\newcommand{\Proj}{\mathbb{P}}

\DeclareMathOperator{\Pf}{Pf}
\DeclareMathOperator{\codim}{codim}
\DeclareMathOperator{\Imago}{Im}
\DeclareMathOperator{\LT}{LT}
\DeclareMathOperator{\rk}{rk}
\DeclareMathOperator{\Sing}{Sing}
\DeclareMathOperator{\Cl}{Cl}
\newcommand{\Span[1]}{\left< #1 \right>}

\usepackage{changepage}

\usepackage{makecell}
\usepackage[all]{xy}

\allowdisplaybreaks
\usepackage{mathpazo}
\usepackage{fancyhdr}
\usepackage{pgf,tikz}
\usetikzlibrary{arrows}
\usepackage{bm}
\usepackage{tikz}
\usepackage{lscape}
\usepackage{mathtools}

\usepackage[all]{xy}

\usepackage[colorlinks]{hyperref}

\usepackage{afterpage}
\usepackage{enumitem}
\usepackage{comment}
\usepackage{accents}
\usepackage{caption}

\usepackage{indentfirst}
\usepackage{color}
\usepackage{graphicx}
\usepackage{newlfont}

\usepackage{amssymb}
\usepackage{amsmath}
\usepackage{latexsym}
\usepackage{amsthm}
\usepackage{mathrsfs}
\usepackage{amscd}

\begin{document}

\title{Sarkisov links for index 1 Fano 3-folds in codimension 4}

\author{Livia Campo}

\address{School of Mathematics \\
	University of Birmingham \\
	Edgbaston \\
	Birmingham, B15 2TT \\
	United Kingdom}

\email{l.campo@bham.ac.uk}



\thanks{The author would like to thank Hamid Ahmadinezhad, Gavin Brown,  
and Miles Reid for the precious conversations and useful comments during the development of this work. The author was supported by EPSRC Doctoral Training Partnership and by EPSRC grant EP/N022513/ held by Alexander Kasprzyk.}

\begin{abstract} 
	We classify Sarkisov links from index 1 Fano 3-folds anticanonically embedded in codimension 4 that start from so-called Type I Tom centres. We apply this to compute the Picard rank of many such Fano 3-folds.
\end{abstract}

\maketitle

\section{Introduction}

The construction of sequences of birational maps linking algebraic varieties to one another is a crucial part in the Minimal Model Program (MMP). 
In the framework of the MMP, the most fundamental such sequences go under the name of Sarkisov links (\cite{Corti95,HaconMcKernan}). In this context, the notions of birational rigidity and pliability for Fano 3-folds, and for Mori fibre spaces (Mfs) more generally, are important, and relate to the uniqueness or otherwise of outputs of the MMP. 
The pliability (see \cite[Definition 1.5(4)]{CortiMella}) is the number of different Mori fibre spaces $W$ that are birational to a given Mfs $X$. If the pliability is 1, then $X$ is said to be birationally rigid; if it is 2 or more, then by \cite{Corti95} it is known that the birational transformation between $X$ and any $W$ can be factorised into a sequence of Sarkisov links. 
The literature often considers Sarkisov links from $X$ according to the codimension of $X$ in its anticanonical embedding. Corti, Pukhlikov, Reid and others (\cite{CPR,CheltsovParkRigidHypersurfaces}) show that quasi-smooth members of the 95 index 1 terminal Fano 3-fold weighted hypersurfaces of \cite{ReidCanonical3folds,IanoFletcher} are birationally rigid. In codimension 2, 19 families of Fano 3-folds are birationally rigid, and 66 are non-rigid (\cite{IskovskikhPukhlikov,OkadaCI,AhmadinezhadZucconiCI}); in codimension 3, Brown and Zucconi prove birational non-rigidity whenever there is a Type I centre \cite{BrownZucconi}. 
Codimension 3 is completed by Ahmadinezhad and Okada \cite{AhmadinezhadOkadaPfaff}, where they prove that an index 1 terminal Fano 3-fold in codimension 3 is birationally rigid if and only if it does not have any Type I or Type II$_1$ centres (this happens for 3 of the 70 Hilbert series). The expectation is that as the codimension increases, rigidity becomes more rare. 

We follow ideas of \cite{CortiMella,BrownZucconi} and we focus on terminal $\Q$-factorial Fano 3-folds in codimension 4 having at least one Type I centre that are listed in the Graded Ring Database \cite{grdb}. 
In particular we examine those deformation families arising from Type I unprojections of pfaffian Fano 3-folds in Tom format (see \cite[Section 3]{T&Jpart1}): we call these \emph{Fano 3-folds of Tom type}.

In our Main Theorem \ref{T outputs of sarkisov links} we give a description of 
birational links for Fano 3-folds of Tom type based on the weights of their ambient space and their basket of singularities, in a similar flavour to the main theorems in \cite{CortiMella,BrownZucconi,CheltsovParkRigidHypersurfaces}. 

Theorem \ref{T outputs of sarkisov links} is is also related to other works in the literature, such as Takagi's \cite{Takagi}, and a comparison with that can be found in Subsection \ref{comparison with Takagi}.
The explicit results are given in detail in \cite{BigTableLinks}. Some important remarks regarding Theorem \ref{T outputs of sarkisov links} are in Section \ref{Setting}.

In Section \ref{section on Pic} we apply Theorem \ref{T outputs of sarkisov links} to compute the Picard rank of some Fano 3-folds in codimension 4.

\section{The Main Theorem} \label{Setting}
We work over the field of complex numbers $\C$. A Fano 3-fold is a normal projective 3-dimensional variety $X$ with ample anticanonical divisor $-K_X$ and at worst terminal singularities.
\begin{definition}
	The \emph{Fano index} of a Fano 3-fold X is defined to be
	\begin{equation*}
	\iota_X \coloneqq \max \{ q \in \Z_{\geq 1} \; : \; -K_X = q A \text{ for some } A \in \Cl(X) \} \; .
	\end{equation*}
\end{definition} 
Our focus will be on those having Fano index $\iota_X = 1$ and codimension 4 in their total anticanonical embedding (cf \cite[Section 1]{AltinokBrownReidK3,T&Jpart1}). A complete description of Type I unprojections is provided in \cite{T&Jpart1}, giving a tool to produce families of Fano 3-folds in codimension 4. 
These realise 115 of the possible Hilbert series, and present at least two distinct deformation families of quasi-smooth Fano 3-folds for each, called \emph{Tom} and \emph{Jerry}. 
By construction, all such Fano varieties have at least one Type I centre. The list of all possibilities for Hilbert series can be found on the Graded Ring Database \cite{grdb}: each is identified with an ID number preceded by \#.
\begin{definition} \label{Tom type def}
	Let $X$ be a codimension 4 index 1 Fano 3-fold $X$ listed in the table \cite{TJBigTable}. We say $X$ is \textit{of Tom Type} if it is obtained as Type I unprojection of the codimension 3 pair $Z \supset D$ in a Tom family (see \cite{T&Jpart1,PapadakisReidKM} for background and examples; see \ref{unprojection setup} for details). 
	The image of $D \subset Z$ in $X$ is called \textit{Tom centre}: it is a cyclic quotient singularity $p \in X$. In the unprojection setup $D \subset Z$, $D$ is a complete intersection of four linear forms of weight $d_1, \dots, d_4$: we refer to $d_1, \dots, d_4$ as the \textit{ideal weights}. Such $X$ of Tom type is said to be \textit{general} if $Z \supset D$ is general in its Tom family.
\end{definition}

We prove the following main theorem. Here $X$ is a $\Q$-factorial Fano 3-fold. At this stage, we do not assume that $X$ is a Mori fibre space. However, the explicit construction that we carry out shows a posteriori that the endpoint of each birational link described in Theorem \ref{T outputs of sarkisov links} is a Mori fibre space.
\begin{theorem} \label{T outputs of sarkisov links}
	Let $X$ be a general codimension 4 Fano 3-fold of Tom type and let $p \in X$ be a Tom centre. Then:
	
	\begin{enumerate}[label=\textbf{(\Alph*)}]
		\item \label{all links for Tom} $X$ admits a birational link to a Mori fibre space $Y \rightarrow S$. The link is initiated by the Kawamata blow-up of $p \in X$.
		Let $d_1 \geq d_2 \geq d_3 \geq d_4$ be the four ideal weights for the Tom centre $p \in X$. In each case the Kawamata blow-up is followed by an algebraically irreducible flop of finitely many smooth rational curves, and proceeds as follows according to $d_1 \geq d_2 \geq d_3 \geq d_4$:
		\begin{enumerate}[label=\textbf{(\roman*)}]
			\item \label{1>2>3>4} \underline{$d_1>d_2>d_3>d_4$}: a composition of two flips, followed by a divisorial contraction $\Phi'$ of $(2,0)$-type to another (non-isomorphic) Fano 3-fold $X'$;
			\item \label{1>2=3>4} \underline{$d_1>d_2=d_3>d_4$}: a flip (missed in cases \#1218 and \#1413) followed by a divisorial contraction $\Phi'$ of $(2,1)$-type to another Fano 3-fold $X'$;
			\item \label{1=2>3>4} \underline{$d_1=d_2>d_3>d_4$}: two simultaneous flips, followed by a divisorial contraction $\Phi'$ of $(2,0)$-type to another Fano 3-fold $X'$;
			\item \label{1>2>3=4} \underline{$d_1>d_2>d_3=d_4$}: a a composition of two hypersurface flips, followed by a del Pezzo fibration: $\Phi'$ is of $(3,1)$-type;
			\item \label{1=2>3=4} \underline{$d_1=d_2>d_3=d_4$}: two simultaneous flips followed by a del Pezzo fibration: $\Phi'$ of $(3,1)$-type;
			\item \label{1>2=3=4} \underline{$d_1>d_2=d_3=d_4$}: a toric flip (missed in case \#6865) to a conic bundle: $\Phi'$ is of $(3,2)$-type;
			\item \label{1=2=3>4} \underline{$d_1=d_2=d_3>d_4$}: a divisorial contraction $\Phi'$ of $(2,1)$-type to another Fano 3-fold $X'$;
			\item \label{1=2=3=4} \underline{$d_1=d_2=d_3=d_4$}: a conic bundle over a quadric surface in $\Proj^3$: $\Phi'$ is of $(3,2)$-type.
		\end{enumerate}
		\item \label{X not bir rigid, X not iso to X'} In every birational link in \ref{all links for Tom}, the resulting Mfs $Y \rightarrow S$ is not isomorphic to $X$. 
		
		\item \label{if Pic=1 then X bir rigid} If in addition the Picard rank of $X$ is $\rho_X=1$, then the link produced in \ref{all links for Tom} is a Sarkisov link, and so $X$ is not birationally rigid.
	\end{enumerate}
\end{theorem} 
The notation on fibrations and divisorial contractions in the above theorem is: $(m,n)$ where $m$ is the dimension of the exceptional locus of $\Phi'$ in $Y$ (where applicable) and $n$ is the dimension of its image. For instance, if $\Phi'$ is of $(2,0)$-type it contracts a $w\PP^2$ to a point in $X'$.

Note that the flip in case \ref{1=2>3>4} and \ref{1=2>3=4} is algebraically irreducible: in this situation, we say that we have two \textit{simultaneous flips}. We expand this in Remark \ref{simultaneous flips}. Moreover, Theorem \ref{T outputs of sarkisov links} does not apply to the Fano 3-folds considered in \cite{brownP2xP2} (see Remark \ref{on P2xP2}).

\section{The input data}

\subsection{Construction and notation} \label{Construction}

The Definition \ref{def Sarkisov link} of Sarkisov link stems from the notion of 2-ray game, especially in the context of toric varieties (see \cite{BrownZucconi} for a description in terms of graded rings). 
A birational link for a codimension 4 Fano 3-fold $X \ni p$ is partially subject to the behaviour of the link for its ambient space $w\Proj^7 \supset X$. 
Call $x_1,x_2,x_3,y_1,y_2,y_3,y_4,s$ the coordinates of $w\Proj^7 = \Proj^7(a,b,c,d_1,d_2,d_3,d_4,r)$, and suppose to blow up the cyclic quotient singularity at $p=P_s \in w\Proj^7$. Call this blow-up $\F_1$: this is a rank 2 toric variety whose bi-grading defining the $\C^\times \times \C^\times$ action on $\F_1$ we will deduce below. 
In toric geometry (cf \cite{CoxToricVarieties}), this corresponds to adding a new lattice vector $\rho_t$ to the 1-skeleton of $w\Proj^7$ given by the 
lattice vectors $\rho_s, \rho_{x_i}, \rho_{y_j}$ that satisfy the relation
\begin{equation*}
r \rho_s + a\rho_{x_1} + b\rho_{x_2} + c\rho_{x_3} + \sum_{j=1}^{4} d_j \rho_{y_j} = 0 \; .
\end{equation*}
The new vector $\rho_t$ must be inside the fan constituted by the convex cone $\sigma_s \coloneqq \langle\rho_{x_1}, \rho_{x_2}, \rho_{x_3},$ $\rho_{y_1}, \rho_{y_2}, \rho_{y_3}, \rho_{y_4}\rangle$; that is, an integer multiple of $\omega \rho_t$ of $\rho_t$ is the integer positive sum of all rays other than $\rho_s$: there are many possible choices to choose the coefficients for this positive sum, and we will identify a particular one. For $\omega$, $\omega_i$, $\delta_j > 0$ and $i \in \{1,2,3\}$, $j \in \{1,2,3,4\}$, the relation involving $\rho_t$ is
\begin{equation}
-\omega \rho_t + \sum_{i=1}^{4} \omega_i\rho_{x_i} + \sum_{j=1}^{4} \delta_j \rho_{y_j} = 0 \; .
\end{equation}
In other words, $\F_1$ is the variety with Cox ring $\C[t, s, x_1, x_2, x_3, y_1, y_2, y_3, y_4]$ having the grading and the irrelevant ideal shown below. In the language of the graded Cox rings, the bottom weights of the bi-grading of $\F_1$ are the coefficient of the rays in the definition of $\rho_t$. Since $\rho_s$ does not appear in the expression for $\rho_t$, its bottom weight is 0. Thus the bi-grading of $\F_1$ looks like
\begin{equation} \label{non well formed scroll}
\left(
\begin{array}{c c | c c c c c c c}
t & s & x_1 & x_2 & x_3 & y_1 & y_2 & y_3 & y_4 \\
0 & r & a & b & c & d_1 & d_2 & d_3 & d_4 \\
-\omega & 0 & \omega_1 & \omega_2 & \omega_3 & \delta_1 & \delta_2 & \delta_3 & \delta_4
\end{array}
\right)
\end{equation}
and its irrelevant ideal is $(t,s) \cap (x_1,x_2,x_3,y_1,y_2,y_3,y_4)$, as indicated by the vertical bar between $s$ and $x_1$. We will determine the values of the bottom weights $\omega, \omega_1,\omega_2,\omega_3, \delta_1, \dots, \delta_4$ later in this Section (also refer to the appendix of \cite{BCZ} for further details). Note that this is not well-formed: we come back to this later.

The 2-ray game for $w\Proj^7$ is determined by the ray-chamber structure of the Mori cone of $\F_1$. Each bi-degree in \eqref{non well formed scroll} represent a character of the $\C^* \times \C^*$ action on $\F_1$ which correspond to the rays $\rho_s, \rho_{x_i}, \rho_{y_j}$; these in particular represent the linear systems associated to each of the coordinates of $w\Proj^7$. This induces the ray-chamber structure in Figure \ref{Mori cone basic}.
\begin{figure}[htb!] 
	\begin{tikzpicture}
	\draw[thick] (0,0) -- (0,1.5);
	\node [above left] at (0,1.5) {$t$};
	
	\draw[thick] (0,0) -- (1,1.5);
	\node [above left] at (1,1.5) {$s$};
	
	\draw[thick] (0,0) -- (1.75,1.25);
	\node [above right] at (1.75,1.25) {$x_1$};
	
	\draw [fill] (1.8,0.3) circle [radius=0.02];
	\draw [fill] (2,0.3) circle [radius=0.02];
	\draw [fill] (2.2,0.3) circle [radius=0.02];

	\draw[thick] (0,0) -- (2.1,-0.5);
	\node [below right] at (2.1,-0.5) {$y_3$};
	
	\draw[thick] (0,0) -- (1.5,-1);
	\node [below right] at (1.5,-1) {$y_4$};
	\end{tikzpicture}
	\caption{The Mori cone of $\F_1$} \label{Mori cone basic}
\end{figure}
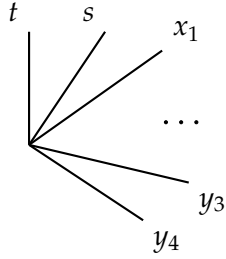

The variation of GIT on $\F_1$ corresponds to the wall-crossing in the picture above. This induces a 2-ray game for $w\Proj^7 \ni p$. Given $X \ni p$ of a Tom-type Fano 3-fold and of a Type I centre $p \in X \subset w\Proj^7$,  we want to embed the 2-ray game for $X \ni p$ into the 2-ray game for $w\Proj^7, \ni p$: this is achieved by finding the appropriate weights $\omega, \omega_1,\omega_2,\omega_3, \delta_1, \dots, \delta_4$ for the grading of the toric variety $\F_1$. 

The objects of the \textit{Mori category} are projective, $\Q$-factorial terminal 3-folds. 
The Fano 3-folds of Definition \ref{Tom type def} are in the Mori category.

\begin{definition} \label{def Sarkisov link}
	A \textit{Sarkisov link} for $X \ni p$ is a birational map between the Mori fibre spaces $X \rightarrow S$ and $X' \rightarrow S'$ that factors as
	\begin{figure}[htb!]	
		\begin{equation} \label{Sarkisov link notation}
		\xymatrix@C=15pt{
			& Y_1 \ar[dl]_\Phi \ar[dr]^{\alpha_1} \ar@{-->}[rr]^{\Psi_1} & & Y_2 \ar[dl]_{\beta_1} \ar[dr]^{\alpha_2} \ar@{-->}[rr]^{\Psi_2} & & Y_3 \ar[dl]_{\beta_2} \ar[dr]^{\alpha_3} \ar@{-->}[rr]^{\Psi_3} & & Y_4 \ar[dl]_{\beta_3} \ar[dr]^{\Phi '} &  \\
			w\Proj \supset X & & Z = Z_1 & & Z_2 & & Z_3 & & X' \subset w\Proj'
		}
		\end{equation} 
	\end{figure}
	
	The birational maps $\Psi_1,\Psi_2,\Psi_3$ are isomorphisms in codimension 1, that is, antiflips, flops, flips in this order (cf \cite[Remark 3.5]{BlancCheltsovDuncanProkhorov}). The map $\Phi$ is a divisorial extraction, and $\Phi'$ can be either a divisorial contraction or a fibration (del Pezzo fibration or conic bundle, in which case the second Mori fibre space is $Y_4 \rightarrow w\Proj'$). Call $\mathbb{G}_i$ the image of $\alpha_i$ (or $\beta_i$), and $Z_i$ the image of $\alpha_i$ restricted to $Y_i$; this is the same as the image of the restriction of $\beta_i$ to $Y_{i+1}$. A Sarkisov link takes place in the Mori category if it satisfies the properties listed in Definition 2.2 of \cite{BrownZucconi}.
\end{definition}
This sets our nomenclature. By a little abuse of notation, we call the coordinates of each $\F_i$ in the same way. Following \cite{BrownZucconi}, each ray of the ray-chamber structure is associated to the linear system defined by the bi-degree of the variable(s) generating it and induces a map of toric varieties. 
Each ray corresponds to one of the toric varieties in the bottom row of the 2-ray game (the ambient spaces of the $Z_i$) in \eqref{Sarkisov link notation}, while each chamber corresponds to one of the $\F_i$ at the top row of \eqref{Sarkisov link notation}. Transitioning from one chamber to another adjacent chamber performs the isomorphism $\Psi_i \colon \F_i \rightarrow \F_{i+1}$ in codimension 1. Approaching the ray in between the two chambers from one side or another indicates the two maps $\alpha_i \colon \F_i \rightarrow \mathbb{G}_i$ and $\beta_i \colon \F_{i+1} \rightarrow \mathbb{G}_i$ (defined by the same linear system). In the language of Geometric Invariant Theory, this is a variation of GIT quotient on $\F_1$.

\subsection{The unprojection setup: construction of $X$} \label{unprojection setup}

The starting point to construct $X$ is the following type of data, coming from \cite{T&Jpart1,grdb}.

\begin{itemize}
	\item A fixed projective plane $D \coloneqq \Proj^2(a,b,c) \subset \Proj^6(a,b,c, d_1, \dots, d_4)$ with coordinates $x_1, x_2, x_3,$ $y_1, \dots,y_4$ respectively and $d_1 \geq d_2 \geq d_3 \geq d_4$. So $D$ is defined by the ideal $I_D \coloneqq \Span[y_1,y_2,y_3,y_4]$.
	\item A family $\mathcal{Z}_1$ of codimension 3 Fano 3-folds $Z \subset w\Proj^6$, each defined by maximal pfaffians of a skew-symmetric $5 \times 5$ syzygy matrix $M$ whose entries $(a_{i,j})$ have weights 
	\begin{equation*}
	\left(
	\begin{array}{c c c c}
	m_{1,2} & m_{1,3} & m_{1,4} & m_{1,5} \\
	& m_{2,3} & m_{2,4} & m_{2,5} \\
	& & m_{3,4} & m_{3,5} \\
	& & & m_{4,5}
	\end{array}
	\right) \; .
	\end{equation*}
\end{itemize}
Here we use the notation of \cite{T&Jpart1} for skew-symmetric matrices: we omit the principal diagonal, whose entries are all zero, and the lower-left triangle, which is the symmetric of the upper-right triangle with opposite signs. The matrix $M$ is graded, i.e. ~each of its entries is occupied by a polynomial in the given degree. A list of the grading of $M$ is in \cite{TJBigTable}.

The plane is a divisor $D \cong \Proj^2_{x_1,x_2,x_3} (a,b,c)$ of $Z_1 \in \mathcal{Z}_1$ if the equations of the latter are the maximal pfaffians of a matrix $M$ in either Tom or Jerry format.

\begin{definition}[\cite{T&Jpart1}, Definition 2.2]\label{Tom definition}
	A $5 \times 5$ skew-symmetric matrix $M$ is in Tom$_k$ format if and only if each entry $a_{i,j}$ for $i,j \neq k$ is in the ideal $I_D$.
\end{definition}

Not all possible formats can be realised: \cite{TJBigTable} records exactly the data of $D \subset w\Proj^6$, the weights of the syzygy matrix, and the successful formats (and why the others fail). Each of them corresponds to a distinct deformation family of $Z_1$ (cf \cite{T&Jpart1}): there can be at most four different deformation families, with at most two realised as Tom formats, and at most two realised as Jerry formats. In this paper we only focus on the Tom case: the aim is to construct $M$ in this general setting by filling its entries with homogeneous polynomials in the $x_i$ and $y_j$ subject to the Tom constraints. 
It is often possible to place some of the variables in a matrix position having the same degree. The following lemma highlights a key feature of $M$, that is, the presence of certain quasilinear monomials in the ideal variables. It is a direct observation on the weights $m_{k,l}$ of $M$. By generality of $Z_1$, $y_j$ and $x_j$ appear linearly in suitable entries. This is in a similar spirit to \cite[Section 3]{BrownZucconi}.

\begin{lemma} \label{quasilinearity of entries}
	Let $Z_1 \supset D$ be a general member of a Tom$_i$ family in \cite{TJBigTable} where $i \in \{ 1, \dots, 5\}$. Then there are at least three entries $a_{k,l}$ of $M$  with $k \not = i$, $l \not = i$ such that $d_j = m_{k,l}$, that is, $y_j$ appears linearly in  $a_{k,l}$. Except for \#12960 in \cite{grdb}, there is an entry $a_{k,l}$ of $M$ with $k = i$ or $l = i$ such that $m_{k,l}$ is equal to $a,b$, or $c$, i.e. linear in at least one of the orbinates $x_j$.
\end{lemma}

Once this set-up in codimension 3 is done and the equations of $Z_1$ are found, the unprojection of $Z_1$ at the divisor $D$ is a birational transformation that produces a new Fano 3-fold $X \subset w\Proj^7$ in codimension 4. In particular, $X$ inherits the five pfaffian equations of $Z_1$ and gains four extra equations from the unprojection process, to which we will refer as unprojection equations in the rest of this paper. The unprojection equations are of the form $s y_i = g_i(x_1,x_2,x_3,y_1,y_2,y_3,y_4)$ where $s$ is the additional coordinate of $w\Proj^7$ and the right-hand side is a homogeneous polynomial of the same degree as $s y_i$. In the unprojection the divisor $D \subset Z_1$ is contracted to the Type I centre $P_s \in X$. In this paper we study birational links from $X \ni P_s$. In Appendix \ref{appendix} we present a brief summary about the explicit construction of the unprojection equations based on \cite{PapadakisComplexes}.

\subsection{The bi-grading of $\F_1$}

Consider $X \ni P_s$. To perform a blow-up $\Phi \colon \F_1 \rightarrow w\Proj^7$ at $P_s$ we choose a suitable grading $\omega, \omega_1,\omega_2,\omega_3, \delta_1, \delta_2, \delta_3, \delta_4$ for $\F_1$ in \eqref{non well formed scroll}. We follow a similar method to \cite{AhmadinezhadZucconiCircles}. Recall the following theorem.
\begin{theorem}[Kawamata blow-up, \cite{kawamata1996divisorial}] \label{kawamata thm}
	Let $X$ be a 3-fold, and $p \in X$ a terminal cyclic quotient singularity $\tfrac{1}{r}(a,b,c)$. Suppose that $\phi \colon (E\subset Y) \rightarrow (\Gamma \subset X)$ is a divisorial contraction with $p \in \Gamma$ and $Y$ terminal. Then, $\Gamma = \{p\}$ and $\phi$ is the weighted blow-up of $p$ with weights $(a,b,c)$ and therefore the exceptional divisor is $E \cong \Proj(a,b,c)$.
\end{theorem}

The blow-up map $\Phi$ is defined by the linear system $\left| \mathcal{O} \big( \substack{\scriptscriptstyle{1}\\\scriptscriptstyle{0}} \big) \right|$. Explicitly,
\begin{align*}
\Phi \; \colon \; \F_1 &\longrightarrow w\Proj^7 \\
\left(t,s, x_1, x_2, x_3, y_1, y_2, y_3, y_4\right) &\longmapsto \left( t^{\tfrac{\omega_1}{\omega}}x_1, t^{\tfrac{\omega_2}{\omega}}x_2, t^{\tfrac{\omega_3}{\omega}}x_3, t^{\tfrac{\delta_1}{\omega}}y_1, t^{\tfrac{\delta_2}{\omega}}y_2, t^{\tfrac{\delta_3}{\omega}}y_3, t^{\tfrac{\delta_4}{\omega}}y_4, s \right) \; .
\end{align*}

The blown-up point is the cyclic quotient singularity of index $r$ at $P_s$, so $\omega = r$. In \cite{BrownZucconi} it is shown that the exceptional locus $E$ of $\Phi$ is given by the vanishing of the coordinates $y_1,y_2,y_3,y_4$; thus, $E \cong \Proj^2(\omega_1,\omega_2,\omega_3)$. In order for $\Phi$ to be a Kawamata blow-up the weights $\omega_1, \omega_2, \omega_3$ must be $a, b, c$ respectively.

The equations of $X$ come into play to determine the value of the $\delta_j$. The key point is the definition of the Fano 3-fold $Y_1$. In the pull-back $\Phi^*(X)$, every monomial in each equation of $X$ picks up a suitable power of $t$. We aim at defining $Y_1$ as the saturation over $t$ of the total pull-back of $X$ (see Definition \ref{def of Y1}). On the other hand, we want to embed the link for $(X,P_s)$ into the link for $(w\Proj^7,P_s)$ in such a way that the birational transformations to which $\F_1$ is subject restrict to $Y_1$. Thus, we want the leading terms of the unprojection equations to be $s y_j$, as opposed to $s y_j t^\tau$ for an exponent $\tau >1$. We give a constructive definition of the $\delta_j$ starting with $\delta_4$: the analysis for $\delta_1$, $\delta_2$, and $\delta_3$ is done analogously. The fourth unprojection equation is of the form $sy_4 = g_4(x_1, x_2, x_3, y_1, y_2, y_3, y_4)$, where $g_4$ is a homogeneous polynomial of degree $r + d_4$. Its pull-back via $\Phi$ is 
\begin{equation*}
t^{\tfrac{\delta_4}{r}} s y_4 = g_4 \left( t^{\tfrac{a}{r}} x_1, t^{\tfrac{b}{r}} x_2, t^{\tfrac{c}{r}} x_3, t^{\tfrac{\delta_1}{r}} y_1, t^{\tfrac{\delta_2}{r}} y_2, t^{\tfrac{\delta_3}{r}} y_3, t^{\tfrac{\delta_4}{r}} y_4 \right) \; .
\end{equation*} 
By construction, every monomial in $g_4$ picks up a $t$ factor because there is no pure monomial in $s$ in $g_4$. Define $h_4$ to be the polynomial constituted by all the monomials of $g_4$ containing $y_4$, except for the term $s y_4$. For $g'_4 \coloneqq g_4 - h_4$ and $\kappa$ the minimum exponent that can be factorised from $h_4$, the equation above becomes
\begin{equation} \label{unproj eqn with g'}
t^{\tfrac{\delta_4}{r}} \left(s y_4 + t^{\tfrac{\kappa-\delta_4}{r}} h_4\right) = g'_4 \left( t^{\tfrac{a}{r}} x_1, t^{\tfrac{b}{r}} x_2, t^{\tfrac{c}{r}} x_3, t^{\tfrac{\delta_1}{r}} y_1, t^{\tfrac{\delta_2}{r}} y_2, t^{\tfrac{\delta_3}{r}} y_3 \right) \; .
\end{equation}

\begin{lemma} \label{delta_j > d_j}
	It holds that $\delta_4 \geq d_4$.
\end{lemma}
\begin{proof}
	By construction of $\rho_t$, each $\delta_j$ is a strictly positive integer. We divide this proof in different cases according to the different types of monomials in $g_4$. We indicate by $\underline{x}^l$ the multiplication of pure powers of $x_1$, $x_2$ and $x_3$, not necessarily all together, with different multiplicities, summarised by the multi-index $l$ at the exponent, and similarly for $\underline{y}^{l'}$. In the following, $l$ and $l'$ vary from case to case. Monomials $\underline{x}^l$ with $|l| = \deg(g_4)=r + d_4$ pick up a $t$ factor with exponent $k = |l|= r + d_4$ in the pull-back. Monomials $\underline{x}^l \underline{y}^{l'}$ with $|l+l'|=\deg(g_4)=r + d_4$ pick up a $t$ factor with exponent $k \geq |l+l'| = r + d_4$ in the pull-back because $\delta_1, \delta_2, \delta_3 \geq 1$. Monomials $\underline{x}^l y_4^\lambda$ with $l+\lambda=\deg(g_4)=r + d_4$ pick up a $t$ factor with power $k \geq l + \lambda \delta_4 \geq r + d_4$. Monomials $\underline{y}^{l'} y_4^\lambda$ with $l'+\lambda=\deg(g_4)=r + d_4$ pick up a $t$ factor with power $k \geq l' + \lambda \delta_4 \geq r + d_4$. Monomials $\underline{x}^l \underline{y}^{l'} y_4^\lambda$ with $l+l'+\lambda=\deg(g_4)=r + d_4$ pick up a $t$ factor with power $k \geq l + l' + \lambda \delta_4 \geq r + d_4$. In conclusion, every monomial in $g_4$ picks up a $t^k$ factor with $ k \geq (r + d_4)/r$: then, $\delta_4 \geq d_4$.
\end{proof}
Then, for $\tau_l$ positive integers and $m_l$ monomials of $g'_4$, the pullback of the unprojection equation for $y_4$ is 
\begin{equation*}
t^{\tfrac{\delta_4}{r}} \left(s y_4 + t^{\tfrac{\kappa}{r}} h_4\right) = t^{\tfrac{\tau_1}{r}}m_1 + \dots + t^{\tfrac{\tau_{k_4}}{r}}m_{k_4} \; .
\end{equation*}
\begin{definition} \label{def delta j}
	Define $\delta_4$ as $\delta_4 \coloneqq \min \{\tau_l \; : \; 1 \leq l \leq k_4\}$.
\end{definition}
Since $g'_4$ does not contain $y_4$, $\delta_4$ is well-defined. The definition of $\delta_1$, $\delta_2$, and $\delta_3$ is analogous. The grading for $\F_1$ that we just obtained might not be well-formed (cf \cite{HamidPliabilityCox}), but a manipulation on the rows of \eqref{non well formed scroll} makes it well-formed \eqref{scroll for tom}.
\begin{proposition} \label{shape of F tom}
	Let $X$ be a codimension 4 index 1 Fano 3-fold of Tom type. Then the Kawamata blow-up of $X$ at the Tom centre $P_s$ is contained in a rank 2 toric variety $\F_1$ with grading
	\begin{equation} \label{scroll for tom}
	\left(
	\begin{array}{c c | c c c c c c c}
	t & s & x_1 & x_2 & x_3 & y_1 & y_2 & y_3 & y_4 \\
	0 & r & a & b & c & d_1 & d_2 & d_3 & d_4 \\
	1 & 1 & 0 & 0 & 0 & -1 & -1 & -1 & -1
	\end{array}
	\right) \; .
	\end{equation}
\end{proposition}
Note that $x_1,x_2,x_3$ generate the same linear system and, therefore, the same ray in the ray-chamber structure of $\F_1$. Since the Fano index of $X$ is 1, then one of the $x_i$ has weight 1. To fix ideas, let the weight of $x_1$ be 1. 
To prove Proposition \ref{shape of F tom} we need the following
\begin{lemma} \label{existence of mon tom}
	Let $Z$ be a codimension 3 Fano 3-fold defined by pfaffians of a $5 \times 5$ skew-symmetric matrix $M$ in Tom format. Consider the Type I unprojection of $Z$ at a divisor $D$. Then each unprojection equation contains at least one monomial purely in $x_1,x_2,x_3$.
\end{lemma}

\begin{proof}
	Since $x_1$ has weight 1, using the notation in Appendix \ref{Papadakis' algorithm}, $p_j$ contains a monomial of the form $x_1^{\deg(p_j)}$. There are different possibilities to fill the ideal entries $a_{k,l}$. If an ideal entry has the same weight as of one of the $y_j$, then it contains such ideal variable linearly, i.e $\alpha^j_{k,l}$ is constant. Otherwise, it contains multiplications of $y_j$ by the $x_i$, that is $\alpha^j_{k,l}$ is a polynomial containing a term in the $x_i$. We assume this without loss of generality.	Therefore, each $N_j$ has at least one entry that is either a constant or a monomial in the $x_i$.
	
	Since the vector of the $g_j$ is independent on the choice of $i$ in \eqref{H_i independent on i}, it is possible to consider only $\tfrac{H_1}{p_1}$. Therefore, we exclude from the calculation of $g_j$ all $\Pf_1(N_j)$, that is, all $\Pf_i(N_j)$ involving the top row of the matrices $N_j$, which are the ones containing pure terms in $x_1,x_2,x_3$. Thus, each entry of $Q$ in row 2, 3, and 4 contains a polynomial purely in $x_1,x_2,x_3$. The same holds for the other $g_i$.
\end{proof}

\begin{proof}[Proof of Proposition \ref{shape of F tom}]	
	By Lemma \ref{existence of mon tom}, each unprojection equation contains at least one monomial in the orbinates of $P_s$. By the proof of Lemma \ref{delta_j > d_j}, such monomial realises the minimum value of Definition \ref{def delta j}. Thus, by Lemma \ref{existence of mon tom}, we choose $\delta_4 = r + d_4$: so, $\delta_4$ is equal to the degree of $g_4$. In turn, we can apply this same strategy to $\delta_1$, $\delta_2$, $\delta_3$: the power of $t$ gained by each $y_j$ factor is greater or equal to $d_j/r$, so $\delta_j \geq d_j$. Since $\delta_j = r + d_j$ for each $j \in \{ 1,2,3,4\}$, the order in which we determine the $\delta_j$ is unimportant. The weights in \eqref{scroll for tom} follow by simple manipulation of the rows of \eqref{non well formed scroll} with the grading we just defined: subtracting the second row to the third row of \eqref{non well formed scroll} and dividing the third row by $-r$ we obtain an isomorphic rank 2 toric variety whose Cox ring is given by \eqref{scroll for tom}.
\end{proof}

\subsection{The Kawamata blow-up of a Fano: equations of the blow-up $Y_1$}

As anticipated above, we define the blow-up $Y_1$ of $X$ at $P_s$ as the following.
\begin{definition} \label{def of Y1}
	The ideal of $Y_1 \subset \F_1$ is the saturation over $t$ of the ideal of $\Phi^*(X)$.
\end{definition}
This motivates the construction of the bottom weights of $\F_1$ made in Definition \ref{def delta j}. In relation to the 1-skeleton in \eqref{scroll for tom}, $\Phi$ and $\alpha_1$ are given by the linear systems $\left| \mathcal{O} \big( \substack{\scriptscriptstyle{r}\\\scriptscriptstyle{1}} \big) \right| $, $\left| \mathcal{O} \big( \substack{\scriptscriptstyle{1}\\\scriptscriptstyle{0}} \big) \right| $ respectively. The next statements make Definition \ref{def of Y1} more manoeuvrable and explicit.

\begin{proposition} \label{Phi and alpha1 proportional}
	The pull-back of the pfaffian equations via $\Phi$ and via $\alpha_1$ are equal up to a $t$ factor.
\end{proposition}
More precisely, the evaluation of $\Pf_i(M)$ at the defining monomials of $\Phi$ is proportional by a $t$ factor to the evaluation of $\Pf_i(M)$ at those of $\alpha_1$.

\begin{proof}
	We prove that $t^{-\frac{1}{r}} \Phi = \alpha_1$. The map $\alpha_1 \colon \F_1 \rightarrow w\Proj^6$ is 
	\begin{equation*}
	\left(t,s, x_1, x_2, x_3, y_1, y_2, y_3, y_4\right) \longmapsto \left( x_1, x_2, x_3, t y_1, t y_2, t y_3, t y_4 \right) \; .
	\end{equation*}
	Consider a variable $w$ of $\F_1$ among $x_1, x_2, x_3,$ $y_1, y_2, y_3, y_4$ with bidegree $\binom{\nu_1}{\nu_2}$. Call $\zeta$ the exponent of the $t$ factor that $w$ needs to pick up such that the bidegree of $w t^\zeta$ is proportional to $\binom{r}{1}$. In other words, we need to find $\zeta$ such that $\deg w t^\zeta = \binom{\nu_1}{\nu_2 + \zeta} = \lambda \binom{r}{1}$ for some $\lambda > 0$.
	Since $\lambda = \nu_2 + \zeta$, we have that $\zeta = \tfrac{\nu_1}{r} - \nu_2$. On the other hand, the exponent $\zeta'$ of the $t$ factor needed such that the bidegree of $w t^{\zeta'}$ is proportional to $\binom{1}{0}$ is $\deg w t^{\zeta'} = \binom{\nu_1}{\nu_2 + \zeta'} = \mu \binom{1}{0}$ for some $\mu > 0$.
	Here $\zeta' = - \nu_2$. Thus, $\zeta - \zeta' = \nu_1/r = 1/r \deg_{w\Proj^7} w$. This means that on every variable $x_1, x_2, x_3,$ $y_1, y_2, y_3, y_4$ of $\F_1$ the exponents $\zeta$ and $\zeta'$ differ only by $1/r$. 
\end{proof}

\begin{remark} \label{ideal entries divisible by t}
	If $M$ is in Tom format, it is possible to cancel out from $\alpha_1^*(\Pf(M))$ a $t$ factor with power at least 1. 
\end{remark}
Let $I_X \coloneqq \Span[f_1, \dots, f_5, f_6, \dots, f_9]$ be the ideal of $X$, generated by polynomials $f_i \coloneqq \Pf_i(M)$ for $i \in \{ 1, \dots, 5\}$ and $f_i \coloneqq s y_i - g_i$ for $i \in \{ 6, \dots, 9\}$. Recall that $\Phi$ is expressed with fractional exponents of $t$. For $\Phi^*(X)$ to have equation in a polynomial ring, we write an equivalent expression for $\Phi$ considering its multiplication by a $t^{(r-a)/r}$ factor. Thus,
\begin{align*} \label{Phi with integer exponents}
&t^{\tfrac{r-a}{r}} \cdot \left( t^{\tfrac{a}{r}}x_1, t^{\tfrac{b}{r}}x_2, t^{\tfrac{c}{r}}x_3, t^{\tfrac{\delta_1}{r}}y_1, t^{\frac{\delta_2}{r}}y_2, t^{\tfrac{\delta_3}{r}}y_3, t^{\frac{\delta_4}{r}}y_4, s \right)= \nonumber \\
& \left( t x_1, t^{\tfrac{b(r-a)}{r} + \tfrac{b}{r}}x_2, t^{\tfrac{c(r-a)}{r} + \tfrac{c}{r}}x_3, t^{\tfrac{d_1(r-a)}{r} + \tfrac{\delta_1}{r}}y_1, t^{\tfrac{d_2(r-a)}{r} + \tfrac{\delta_2}{r}}y_2, t^{\tfrac{d_3(r-a)}{r} + \tfrac{\delta_3}{r}}y_3, t^{\tfrac{d_4(r-a)}{r} + \tfrac{\delta_4}{r}}y_4, t^{r-a} s \right)
\end{align*} 
This expression has integer exponents. Call $I_{\Phi^* X} \coloneqq \Span[\Phi^* f_1, \dots, \Phi^* f_5, \Phi^*f_6, \dots, \Phi^*f_9]$. Remark \ref{ideal entries divisible by t} guarantees that, up to a $t$ factor, $\Phi^*$ and $\alpha_1^*$ coincide on the pfaffian equations. Define the following polynomials
\begin{align}
h_1 &\coloneqq \tfrac{\alpha_1^*\Pf_1(M)}{t^2} = \tfrac{\alpha_1^*f_1}{t^2} \; ; & \\
h_i &\coloneqq \tfrac{\alpha_1^*\Pf_i(M)}{t} = \tfrac{\alpha_1^*f_i}{t} \qquad &\text{for } i \in \{ 2, \dots, 5\} \; ; \\
h_i &\coloneqq \tfrac{\Phi^*f_i}{t^{\delta_{i-5} + r-a}} \qquad &\text{for } i \in \{ 6, \dots, 9\} \; ;
\end{align}
and the ideal $I_{Y_1} \coloneqq \left( I_{\Phi^* X} \colon t^\infty \right)$ as the saturation of $I_{\Phi^* X}$ over $t$ as in Definition \ref{def of Y1}.

\begin{lemma}\label{saturation is ideal of Y}
	We have that $I_{Y_1} = \Span[h_1, \dots, h_5, h_6, \dots, h_9]$.
\end{lemma}

\begin{proof}
	For the saturation algorithm we refer to \cite{CoxLittleOSheaIdVarAlgorithms}: we introduce a temporary variable $z$ and define the ideal $ J \coloneqq \Span[ I_{\Phi^* X}, t z -1 ] \subset S \coloneqq R[z]$, where $R\coloneqq \C \left[ t, s, x_1, x_2, x_3, y_1, y_2, y_3, y_4 \right]$. We study explicitly the Gr\"{o}bner basis of $J$ with respect to a complete monomial ordering $\succ$. Then, $\left( I_{\Phi^* X} : t^\infty \right) = J \cap R$ (see \cite[Chapter 4, \S 4]{CoxLittleOSheaIdVarAlgorithms}). The monomial ordering is such that $z$ is the largest, $s$ is the second largest, and the monomials containing the least number of $y_j$ follow. In other words, $\succ$ is defined by
	
	\begin{equation} \label{matrix of monomial order}
	\left(
	\begin{array}{c c c c c c c c c c}
	z & s & x_1 & x_2 & x_3 & y_1 & y_2 & y_3 & y_4 & t \\
	1 & 0 & 0 & 0 & 0 & 0 & 0 & 0 & 0 & 0 \\
	0 & 1 & 0 & 0 & 0 & 0 & 0 & 0 & 0 & 0 \\
	0 & 0 & a & b & c & d_1 -1 & d_2 -1 & d_3 -1 & d_4 -1 & 1 \\
	0 & 0 & a & b & c & d_1 -1 & d_2 -1 & d_3 -1 & d_4 -1 & 0 \\
	0 & 0 & a & b & c & d_1 -1 & d_2 -1 & d_3 -1 & 0 & 0 \\
	0 & 0 & a & b & c & d_1 -1 & d_2 -1 & 0 & 0 & 0 \\
	0 & 0 & a & b & c & d_1 -1 & 0 & 0 & 0 & 0 \\
	\end{array}
	\right) \; .
	\end{equation}
	Consider a polynomial $k$ in which $z$ does not appear. Call $k_1 \coloneqq \LT(k)$ the leading term of $k$ according to $\succ$, so $k = k_1 + k_2$ is the sum of the monomial $k_1$ and the polynomial $k_2 \coloneqq k - k_1$. The least common multiple between the respective leading terms of $t^d k$ and $t z -1$ is $\mathrm{lcm} \left( \LT(t^d k), \LT(t z -1) \right) = t^{d+1} k_1 z$. Then, following \cite{CoxLittleOSheaIdVarAlgorithms}, the S-polynomials for $t^d k$ for some $d \geq 1$ are $S\left( t^d k, t z -1 \right) = t^d k$.
	The leading term of $\Phi^*f_1$ is of the form $\LT(\Phi^*f_1) = t^2 y_{j_1} y_{j_2}$ for certain $j_1, j_2 \in \{ 1, 2, 3, 4\}$. Similarly for $i \in \{ 2, \dots , 5\}$, the leading term is $\LT(\Phi^*f_i) = t x_{j_i} y_{j_i}$ for certain $j_i \in \{ 1, 2, 3\}$ and $j_i \in \{ 1, 2, 3, 4 \}$. For $i \in \{ 6, \dots , 9\}$ instead, $\LT(\Phi^*f_i) = s y_{i-5} t^{\delta_{i-5} + r-a}$. The monomial ordering $\succ$ is designed to identify as biggest the monomials having the lowest exponent of $t$. Therefore, for each $i \in \{ 1, \dots , 9\}$ there is a suitable $d$ such that $\Phi^*f_i = t^d h_i$. So we have that $S\left( \Phi^*f_i, t z -1 \right) + S\left( t^d h_i, t z -1 \right) = t^d h_1$.
	Therefore, the Gr\"{o}bner basis of $J$ is 
	\begin{align*}
	&GB_\succ (\Phi^*f_1, \dots, \Phi^*f_9, t z -1) = \\
	&\left( t^ h_1, t h_2, t h_3, t h_4, t h_5, t^{\delta_{1} + r-a} h_6, t^{\delta_{2} + r-a} h_7, t^{\delta_{} + r-a} h_8, t^{\delta_{4} + r-a} h_9 \right) \cup \{ t z -1 \} \; .
	\end{align*}	
	For $i \in \{ 1, \dots , 9\}$, we have that $\LT(h_i)$ and $t z$ are coprime since the highest common factor $hcf \left( \LT(h_i), t z \right) = 1$. Thus, $	GB_\succ (h_1, \dots, h_9, t z -1) = GB_\succ \left(h_1, \dots, h_9 \right) \cup \{ t z -1 \}$. In conclusion,
	\begin{equation*}
	\begin{aligned}
	\left( \Span[h_1, \dots, h_9] \colon t^\infty \right) &= \Span[GB_\succ (h_1, \dots, h_9, t z -1) \cap R] \\
	&= \Span[GB_\succ (h_1, \dots, h_9)] = \Span[h_1, \dots, h_9] \; . 
	\end{aligned}
	\qedhere
	\end{equation*}
\end{proof}

\section{Description of the links and proof of the Main Theorem} \label{proof of main theorems}

We break down every step of the birational links described in Theorem \ref{T outputs of sarkisov links} and we give a proof of Theorem \ref{T outputs of sarkisov links}. We first mention the following remarks about Theorem \ref{T outputs of sarkisov links}.

\begin{remark} \label{simultaneous flips}
	The flip in case \ref{1=2>3>4} and \ref{1=2>3=4} is algebraically irreducible (that is, its base is irreducible as an algebraic set and its exceptional locus consists of one connected component), but the intersection between its exceptional locus and $Y_2$ consists of two disjoint tubular neighbourhoods, that are either both toric or both hypersurface. In other words, the intersection between $Y_2$ and the contracted locus of the flip is not irreducible, and it is formed of two distinct connected components. 
	
	In \ref{1=2=3>4} the exceptional divisor of $\Phi'$ contracts to an irreducible conic $\Gamma \subset \Proj^2$.
\end{remark}

\begin{remark} \label{on P2xP2}
	This theorem does not consider the Fano 3-folds in $\Proj^2 \times \Proj^2$ format of \cite{brownP2xP2} (often stemming from the second Tom deformation families), as they have Picard rank 2. The Hilbert series \#12960 is one of them, thus it is not described in \ref{1=2=3=4} of Theorem \ref{T outputs of sarkisov links}. 
\end{remark}

Let $X \subset w\Proj^7$ be a general codimension 4 Fano 3-fold of Tom type and $p \in X$ a Tom centre. We first prove part \ref{X not bir rigid, X not iso to X'} of Theorem \ref{T outputs of sarkisov links}. Part \ref{if Pic=1 then X bir rigid} is an immediate corollary of parts \ref{all links for Tom} and \ref{X not bir rigid, X not iso to X'}.

\begin{proof}[Proof of Theorem \ref{T outputs of sarkisov links}, \ref{X not bir rigid, X not iso to X'}]
	Consider a birational link for $X \ni p$ that terminates with a divisorial contraction. Suppose that the endpoint Mori fibre space $Y \rightarrow S$ is a Fano 3-fold $X' \rightarrow S=\{ pt \}$. Let $\mathcal{B}_X$ be the basket of singularities of $X$. It is possible to track $\mathcal{B}_X$ throughout the link to retrieve the basket $\mathcal{B}_{Y_4}$ of $Y_4$. The basket $\mathcal{B}_{X'}$ of $X'$ is a subset of $\mathcal{B}_{Y_4}$; that is, $\mathcal{B}_{X'}$ is $\mathcal{B}_{Y_4}$ minus the cyclic quotient singularities sitting inside the exceptional locus $E':=\mathbb{E}'\cap Y_4$. Moreover, if the determinant 
	\begin{equation} \label{determinant}
	\det\left(
	\begin{array}{c c}
	d_3 & d_4 \\
	-1 & -1
	\end{array}\right) = -1
	\end{equation}
	then $E'$ is contracted to a Gorenstein point $p' \in X'$, which does not contribute to the basket of $X'$. We claim that if $\Phi$ blows up the cyclic quotient singularity of highest degree, neither the flops nor the flips create a new cyclic quotient singularity of that degree. To prove this we refer to the notation used in the proof of Theorem \ref{it is a flip} below. First of all, the flop $\Psi_1$ leaves the basket of $Y_1$ unchanged, so $\mathcal{B}_{Y_1} = \mathcal{B}_{Y_2}$. Now for simplicity, suppose that $P_{y_2} \in Z_2$, that is $\Psi_2$ restricted to $Y_2$ is not an isomorphism. We have that $a,b,c < r$, and, since $\Phi$ blows up the cyclic quotient singularity of highest degree, $r \geq d_1$. In Theorem \ref{it is a flip} we prove that the map $\beta_2$ extracts a weighted $\PP^1$ or $\PP^2$, whose weights are among the following $(d_2-d_1, d_3-d_1,d_4-d_1)$, which are all strictly less than $r$ and $d_1$. Thus, the flip $\Psi_2$ never introduces singularities of degree $d_1$ or higher. A similar argument can be carried out when $d_1=d_2>d_3$. 
	
	On the other hand, if $\Phi$ blows up a cyclic quotient singularity of a lower degree that is, $r < d_1$, then the flips get rid of the one with higher degree, which will not be generated again for the same reason as above. Still referring to the notation used in the proof of Theorem \ref{it is a flip}, $\alpha_2$ contracts a weighted $\PP^1$ or $\PP^2$ with weights among $(d_1,a,b,c)$. The coordinate $y_1$ appears linearly in one entry of the matrix $M$ so locally analytically at $P_{y_1}$ one of the orbinates is eliminated. In addition, if $\alpha_2$ contracts a weighted $\PP^2$ the hypersurface thta is the intersection of $\PP^2$ and $Y_2$ always contains the coordinate point associated to $y_1$. Thus, the cyclic quotient singularities with higher degree are contracted by $\alpha_2$. 
	Therefore, the baskets $\mathcal{B}_X$ and $\mathcal{B}_{X'}$ are different, hence $X \not\cong X'$.  
	
	If the absolute value of the above determinant is greater or equal than 2, the divisorial contraction $\Phi'$ might create a new orbifold singularity with order equal to the absolute value of the determinant. This is a phenomenon that occurred already in \cite[Proposition 3.11]{BrownZucconi}: we refer to the latter for the proof of this fact. If $S$ is either a line or $\Proj^2$, $Y$ is $Y_4$. We conclude that $X$ cannot be isomorphic to $Y$ because their Picard ranks are different.
\end{proof}
The rest of this section is dedicated to proving part \ref{all links for Tom} of Theorem \ref{T outputs of sarkisov links}.  
Following the notation in \ref{Construction}, we call $Y_i$ the push-forward $\Psi_{i*}(Y_{i-1}) \subset \F_i$ of $Y_{i-1}$ via $\Psi_i$. The Cox rings of the $\F_i$ can be naturally identified, as they are isomorphic in codimension 1: similarly holds for the Cox rings of the $Y_i$, for which we may choose the same generators of the quotient ideal. Throughout this paper we identify these rings and these coordinates, for all $\F_i$ and $Y_i$.
\begin{theorem} \label{first step is a flop}
	The first step $\psi_1 \colon Y_1 \dashrightarrow Y_2$ of the birational link for $X$ consists of a number $n$ of simultaneous flops, that is equal to the number of nodes on $D \subset Z_1$.
\end{theorem}
To fix ideas and without loss of generality, assume $X$ is of Tom$_1$ type throughout this proof: then, $\Pf_2(M), \Pf_3(M), \Pf_4(M), \Pf_5(M)$  are linear in the generators of $I_D$ and $\Pf_1(M)$ is quadratic in those. The locus $\mathbb{B}_1 \in \F_2$ extracted by $\beta_1$ is defined by $\{ t=s=0 \}$ (cf \cite{BrownZucconi}), which is isomorphic to a weighted $\Proj^3$. Therefore there is a weighted  $\Proj^3$-bundle over $\Proj^2_{x_1,x_2,x_3} \cong D$. Thus, restricting to $\{ t=s=0 \}$ we have that $\left(\Pf_2, \Pf_3, \Pf_4, \Pf_5 \right)^T = A \cdot \left(y_1, y_2, y_3, y_4 \right)^T$
where $A$ is a $4 \times 4$ matrix defined as $A :=\left( \gamma_i(\Pf_j) \right)_{i=1\dots 4, j = 2 \dots 5}$
and $\gamma_i(\Pf_j) \in \C[x_1,x_2,x_3]$ is the coefficient of $y_i$ in $\Pf_j$. We need the following technical lemma.
\begin{lemma} \label{rank of A}
	For each point $p \in D$ the rank of $A_p := ev_p (A)$ is either 2 or 3.
\end{lemma}
\begin{proof}
	The rank is at least 1. There are six syzygies involving the five pfaffians of $M$; in the notation set in \eqref{general form T1} one of them is $p_1 \Pf_2 + p_2 \Pf_3 + p_3 \Pf_4 + p_4 \Pf_5 = 0$. Therefore, at any point $p \in D$ it is possible to express one of the last four pfaffians in terms of the other three. This means that there are three equations left that are linear on $I_D$. Thus, $\rk(A_p) \leq 3$. 
	The restriction to $D$ kills all the monomials that come out from the non-linear (in $I_D$) terms of $\Pf_2,\Pf_3, \Pf_4, \Pf_5$. Since each of the $\Pf_j$ has at least one of the $y_i$ appearing at least once, then there are at least two linearly independent column vectors in $A$. Therefore, the entries of $A$ are all polynomials in $\C[x_1,x_2,x_3]$, so $\rk(A_p) \geq 2$.	
\end{proof}
\begin{proof}[Proof of Theorem \ref{first step is a flop}]
	We first prove that $\alpha_1$ contracts $n$ smooth rational curves. The locus $\mathbb{A}_1$ contracted by $\alpha_1$ is defined by $\{ y_1 = y_2 = y_3 = y_4 = 0 \}$ by \cite{BrownZucconi}. Since $Z_1$ is in Tom format, $Z_1 \cap \Imago(\alpha_1)$ restricted to $\mathbb{A}_1$ depends only on $x_1,x_2,x_3$, that is, it lies on $D$. Hence, over every node on $D$ there is a $\Proj^1$ having coordinates $t,s$. There are $n$ nodes on $D$ so $\alpha_1$ contracts $n$ lines.
	
	Now we need to prove that $\beta_1$ extracts $n$ lines. Since the locus $\mathbb{B}_1$ is fibred over $D$ with weighted $\Proj^3$ fibres, then at any point $p \in D$, if $ \rk(A_p)=2$ the image of $A$ is a 2-dimensional space in $\Proj^3$, which means that $\beta_1$ contracts a $\Proj^1 \subset \mathbb{B}_1 \cap Y_2$ to $p \in D$. Analogously, if $ \rk(A_p)=3$ the map $\beta_1$ is an isomorphism in a neighbourhood of a point $p' \subset \Proj^1 \subset \mathbb{B}_1$ to $p \in D$. This argument uses only the pfaffian equations of $X$, which contain all the necessary information about the flop. In fact, the unprojection equations do not play any role in the determination of the flop: note that $\mathbb{G}_1$ is a rank 1 toric variety of dimension 10 containing $w\Proj^6 \supset Z_1$. Its coordinates are $\xi_1 \coloneqq x_1$, $\xi_2 \coloneqq x_2$, $\xi_3 \coloneqq x_3$, $\upsilon_1 \coloneqq y_1 t$, $\upsilon_2 \coloneqq y_2 t$, $\upsilon_3 \coloneqq y_3 t$, $\upsilon_4 \coloneqq y_4 t$, $\sigma_1 \coloneqq s y_1$, $\sigma_2 := s y_2$, $\sigma_3 \coloneqq s y_3$, $\sigma_4 \coloneqq s y_4$. The unprojection equations globally eliminate the variable $s$ on $D$. In addition, on $D$, the Jacobian matrix of $Z_1$ is
	\begin{equation*}
	J(Z_1) |_D =
	\left(
	\begin{array}{c|c}
	0 & 0 \\
	\hline
	0 & A
	\end{array}
	\right) \; .
	\end{equation*}
	Therefore we deduce that for each point $p \in D$ we have $\rk(J(Z_1) |_D)_p = \rk(A_p)$.
	
	Lastly, we prove that $\psi_1$ is a flop. We have just shown that $\psi_1$ is an isomorphism in codimension 1, so we study the intersections $-K_{Y_1} \cap \mathbb{A}_1$ and $-K_{Y_2} \cap \mathbb{B}_1$. For both $i$ equal to 1 or 2, $-K_{Y_i}$ is $\{x_1=0\}$. On the other hand, none of the points in $\Sing(Z_1) \subset D$ satisfies the condition $x_1=0$. Hence, $-K_{Y_i} \cdot \Proj^1_{t,s}= 0$ for $i= 1,2$. 
\end{proof}
This proof is independent on the form of the right-hand-side of the unprojection equations: the information about the flop is all encoded in the geometry of $Z_1$, as expected.

We introduce the following configurations for the matrix $M$, which appear frequently in the rest of this paper. The argument holds independently on the Tom format. For some suitable positive integers $\sigma$ and $\tau$, define
\begin{enumerate}[label=\textbf{(\alph*)}]
	\item \label{case A} The entries $a_{2,4}, a_{2,5}, a_{3,4}, a_{3,5}$ all have weight $\pi$. Hence, in order to have homogeneous pfaffians and positive weights, the other weights of $M$ are
	\begin{equation} 
	\left(
	\begin{array}{c c c c}
	\sigma & \sigma & \pi + \sigma - \tau & \pi + \sigma - \tau \\
	\hline
	& \tau & \pi & \pi \\
	& & \pi & \pi \\
	& & & 2 \pi - \tau
	\end{array}
	\right) \; .
	\end{equation} 
	\item \label{case B} The entries $a_{2,5}, a_{3,4}$ both have weight $d_1=d_2$, while $a_{2,4}, a_{3,5}$ are free. Hence, the other weights of $M$ are
	\begin{equation} 
	\left(
	\begin{array}{c c c c}
	\sigma & \pi + \sigma - \upsilon & \pi + \sigma - \tau & 2 \pi + \sigma - \tau - \upsilon \\
	\hline
	& \tau & \upsilon & \pi \\
	& & \pi & 2 \pi- \upsilon \\
	& & & 2 \pi - \tau
	\end{array}
	\right) \; .
	\end{equation} 
\end{enumerate}

\subsection{Proof of \ref{1>2>3>4}} \label{proof of >>>}
The following theorem describes the flip that occurs when crossing the ray $\rho_{y_1}$. An identical argument applies when crossing $\rho_{y_2}$ if $d_1>d_2>d_3$.

\begin{theorem} \label{it is a flip}
	Suppose $d_1 > d_2$ and that $P_{y_1} \in Z_2$. Then, $\psi_2 \colon Y_2 \dashrightarrow Y_3$ is a flip.
\end{theorem}
\begin{proof}
	Localise at the point $P_{y_1} \in Z_2$. So, after a row operation, the grading of $\F_2$ becomes
	\begin{equation*}
	\left(
	\begin{array}{c c c c c | c c c c}
	t & s & x_1 & x_2 & x_3 & y_1 & y_2 & y_3 & y_4 \\
	d_1 & r+d_1 & a & b & c & 0 & d_2-d_1 & d_3-d_1 & d_4-d_1 \\
	1 & 1 & 0 & 0 & 0 & -1 & -1 & -1 & -1
	\end{array}
	\right) \; .
	\end{equation*}
	The exceptional locus of $\alpha_2$ is $\mathbb{A}_2 = \{ y_2 = y_3 = y_4 = 0\}$, that is,
	\begin{equation*}
	\mathbb{A}_2 = \left(
	\begin{array}{c c c c c | c}
	t & s & x_1 & x_2 & x_3 & y_1 \\
	d_1 & r+d_1 & a & b & c & 0 \\
	1 & 1 & 0 & 0 & 0 & -1
	\end{array}
	\right) \cong \Proj^4(d_1,r+d_1,a,b,c)
	\end{equation*}
	and $\alpha_2(\mathbb{A}_2) = P_{y_1}$. 
	To prove that $\psi_2$ is a flip for $Y_2$, we show that the codimension of the intersection $Y_2 \cap \mathbb{A}_2$ is at least 3 in $\Proj^4(d_1,r+d_1,a,b,c)$. The unprojection equation $s y_1 = g_1$ allows one to eliminate $s$ locally above $P_{y_1} \in Z_2$. Thus, for a hypersurface $F$ isomorphic to the weighted $\Proj^3(d_1,a,b,c)$ defined by the unprojection equation relative to $y_1$ in which $y_1$ has been set at 1, we have $Y_2 \cap \mathbb{A}_2 \subset F \subset \Proj^4(d_1,r+d_1,a,b,c)$.
	Hence, $Y_2 \cap \mathbb{A}_2$ has at least codimension 1. By Lemma \ref{quasilinearity of entries}, in one of the pfaffian equations there is a monomial of the form $x_i y_1$, that is, locally at $P_{y_1}$, it is possible to eliminate $x_i$, i.e. $x_i$ can be expressed as a function of the other variables: suppose $i=1$. Thus, $Y_2 \cap \mathbb{A}_2 \subset \Proj^4(d_1,r+d_1,a,b,c)$ has at least codimension 2. From Lemma \ref{existence of mon tom} we deduce that there is another unprojection equation that contains monomials in the $x_i$ and $t$. Therefore,  $Y_2 \cap \mathbb{A}_2 \subset S \subset F \subset \Proj^4(d_1,r+d_1,a,b,c)$ where $S \cong \Proj^4(d_1,b,c)$: so, $Y_2 \cap \mathbb{A}_2$ has at least codimension 3 in $\Proj^4(d_1,r+d_1,a,b,c)$. To prove that the codimension is exactly 3 we need to show that the remaining equations define a curve in $S$, so we need to exclude the case in which they define a single point or the empty set. The vanishing locus of the remaining equations cannot be empty because $P_{y_1} \in Z_2$, so there must be an intersection between $Y_2$ and $\mathbb{A}_2$. In addition, $Y_2 \cap \mathbb{A}_2$ cannot be a single point either for the following reason. Since $X$ is quasi-smooth and $\Q$-factorial, the same holds for $Y_1$. Also $Y_2$ is quasi-smooth, but it is not isomorphic to $Y_1$ because $\beta_2 \colon Y_3 \rightarrow Z_2$ contracts the curve defined by the quadratic pfaffian equation (which is $\Pf_1$ if $M$ is in Tom$_1$ format).  Thus, by $\Q$-factoriality, $\beta_2$ must also contract a curve.
	
	The last thing to check is that the intersection of $-K_{Y_2}$ with $\mathbb{A}_2$ is positive and that the intersection of $-K_{Y_3}$ with $\mathbb{B}_2$ is negative. This is true because $\{ x_1 = 0\} \in \left| \mathcal{O}(-a K_{Y_2}) \right|$ is relatively ample with respect to $\alpha_2$, so it meets every curve positively.
\end{proof}

\begin{theorem} \label{skipped flip thm}
	If the point $P_{y_1} \not\in Z_2$, the restriction to $Y_2$ of the toric flip $\Psi_2 \colon \mathbb{F}_2 \dashrightarrow \mathbb{F}_3$ is an isomorphism $Y_2 \cong Y_3$.
\end{theorem}
\begin{proof}
	The equations of $Z_2$ are the same as $Z_1$, albeit viewed in a different toric variety $\mathbb{G}_2$. If $P_{y_1} \not\in Z_2$ then there exists at least one pfaffian equation that is non-zero when evaluated at $P_{y_1}$. Moreover, $\alpha_2(\mathbb{A}_2) = P_{y_1}$; on the other hand, $\alpha_2(Y_2) = Z_2$. This means that the exceptional locus of the flip at the toric level does not intersect with $Y_2$, i.e. $\mathbb{A}_2 \cap Y_2 = \emptyset$.
\end{proof}

The nature of the weights of $M$ determine whether the hypotheses of either Theorem \ref{it is a flip} or Theorem \ref{skipped flip thm} are verified.

\begin{proposition} \label{skipped flip in >>>}
	Let $X$ be of Tom type. If the weights of $M$ fall in case \ref{case B}, then either the flip with base at $P_{y_1} \in Z_2$ or the flip with base at $P_{y_2} \in Z_3$ is an isomorphism.
\end{proposition}
\begin{proof}
	In case \ref{case B} two ideal entries with the same weight are multiplied in $\Pf_1(M)$. Suppose that $\pi=d_1$. Thus, $y_1$ occupies both the entries $a_{2,5}$ and $a_{3,4}$. From Theorem \ref{saturation is ideal of Y} and since $y_1$ appears linearly in those entries, we deduce that the monomial $y_1^2$ is in the equations of $Y_1$. Therefore, repeating the proof of Theorem \ref{skipped flip thm}, we have that $\Psi_2$ restricted to $Y_2$ is an isomorphism. Same happens for $\pi=d_2$. The weight $\pi$ is never equal to $d_3$ or $d_4$.
\end{proof}
In this argument it is crucial that there is only one ideal generator having weight $d_1$. The concurrent presence of configuration \ref{case A} and of two distinguished ideal generators having the same weight leads to different consequences in \ref{1=2>3>4} and \ref{1=2>3=4}.

Although the majority of the Hilbert series of case \ref{1>2>3>4} falls in configuration \ref{case B}, it also happens that the weights of $M$ are in configuration neither \ref{case A} nor \ref{case B}. In this situation, both $\psi_1$ and $\psi_2$ are flips. In particular, this means that the mobile cone of $\F_1$ coincides with the mobile cone of $Y_1$. 
Theorem \ref{it is a flip} and Proposition \ref{skipped flip in >>>} can be also applied to the crossing of the wall adjacent to $d_2>d_3$.

Consider the rank 2 toric variety $\F_4$ in case \ref{1>2>3>4}, where $d_3>d_4$. The link terminates with a divisorial contraction.

\begin{lemma} \label{last map is div contr}
	Suppose that $\rho_X = 1 $. If $d_3>d_4$, the map $\Phi' \colon \F_4 \rightarrow \mathbb{G}_4$ is a divisorial contraction of $Y_4$ to a Fano 3-fold $X' \subset \Proj' \subset \mathbb{G}_4$.
\end{lemma}
\begin{proof}
	Since  $\rho_X = 1 $, the exceptional divisor $\mathbb{E}'$ of $\Phi'$ is irreducible. Thus, $\rho_{X'} = 1 $ as well. Moreover, $X'$ is projective. In addition, $-K_{X'}$ is ample. Consider a curve $\Gamma$ in $X'$ that is not in the image of $\mathbb{E}'$ via $\Phi'$ and that is not in the image of the union of the right-hand-side contracted loci $\mathbb{B}_i$ of the flips. Such a curve can always be found because the set of curves of $X'$ lying in $\Phi'(\mathbb{E}')$ and the union of the proper transform of the $\mathbb{B}_i$ has codimension 2. The curve $\Gamma$ can be tracked back down to $Y_1$. Now, $-K_{Y_1} = \alpha_1^*(-K_{Z_1})$ and every curve in $Y_1$ is either a flopping curve or strictly positive against $-K_{Y_1}$ and contracted to $Z_1$. So, the divisor $-K_{Y_1}$ is nef and big, that is $Y_1$ is a weak Fano. Thus we have $-K_{X'} \Gamma = -K_{Y_1} \tilde{\Gamma} >0$, where $\tilde{\Gamma}$ is the proper transform of $\Gamma$, and is isomorphic to $\Gamma$.
\end{proof}

\subsubsection{Identifying the end of the link} \label{identifying X'}

By Lemma \ref{last map is div contr}, $\Phi'$ is a divisorial contraction to another Fano. A crucial observation is the following.

\begin{lemma} \label{X' has smaller codim}
	Let $X$ be a Fano 3-fold of Tom type and $p \in X$ a Tom centre such that the link for $X \ni p$ terminates with a divisorial contraction to another Fano 3-fold $X' \subset \Proj'$. Then, $\codim(X') < \codim(X)$.
\end{lemma}
\begin{proof}
	The map $\Phi' \colon \F_4 \rightarrow \mathbb{G}_4$ is induced by the linear system $\big| \mathcal{O} \big( \substack{\scriptscriptstyle{d_3}\\\scriptscriptstyle{-1}} \big) \big|$. The equations of $Y_4$ constitute relations among the new coordinates of $\mathbb{G}_4$, that is, some of the equations of $Y_4$ eliminate (globally) some of the coordinates of $\mathbb{G}_4$. The global elimination of the variable $s'=s y_4^\varsigma$ of $\mathbb{G}_4$, for some exponent $\varsigma$, always happens due to the equation $s y_4 = g_4$.
	This phenomenon may occur for other coordinates too, depending on each specific case. This shows that $\Proj'$ is at most a weighted $\Proj^6$.
\end{proof}

It is also possible to track down the evolution of the basket of singularities of $X$ along the link, in order to deduce the one for $X'$. We refer to the proof of Theorem \ref{T outputs of sarkisov links}, \ref{X not bir rigid, X not iso to X'} for this. Its basket and its ambient space determine the Hilbert series of $X'$ univocally. 
We give an example of how to find $\mathcal{B}_{X'}$ in Section \ref{10985 T1}.

The equations of $X'$ can be found by rewriting the equations of $Y_4$ in terms of the new coordinates of $\mathbb{G}_4$, and by excluding the ones used to perform the global elimination. Usually, $X'$ is a special member in the family associated to its Hilbert series. We show this explicitly in the examples of Section \ref{examples}.

\subsection{Proof of \ref{1>2=3>4}}

Suppose $M$ has weights as in \ref{case B}, only for the two Hilbert series \#1218 and \#1413. For both, the equations of $Y_2$ have a pure monomial in $y_1$ (similarly to the phenomenon described in Theorem \ref{skipped flip thm}). Thus, the following holds.

\begin{theorem} \label{d2 = d3 special case}
	Consider the Hilbert series \#1218, \#1413 and the Fano 3-fold defined by Tom$_1$ for both. Then, their respective birational links evolve as follows: $\psi_1$ is a flop, $\Psi_2$ restricts to an isomorphism $\psi_2$ on $Y_2$, $\phi'$ is a divisorial contraction over $\Proj_{y_2,y_3}^1 \subset X'$.
\end{theorem}
\begin{proof}
	By Theorem \ref{first step is a flop} $\psi_1$ is a flop. The weights of $M$ of the two Hilbert series are as in \ref{case B}. Therefore, $\Pf_1(M)$ contains $y_1^2$. Analogously to Theorem \ref{skipped flip thm}, $\psi_2$ is an isomorphism. The last map is a divisorial contraction to $X'$ by Lemma \ref{last map is div contr}. Note that $\Proj_{y_2,y_3}^1 \subset X'$ in both cases, so there is one divisorial contraction to $\Proj_{y_2,y_3}^1$.
\end{proof}

None of the other Hilbert series in \ref{1>2=3>4} comes from $M$ with \ref{case B} weights. In this instance, the flip $\psi_2$ is performed by $Y_2$ too, and it is followed by a divisorial contraction to $X'$.

\begin{theorem}
	Let $Z_1$ be defined by $M$ in Tom format with weights not in \ref{case B}. Then the birational link for $X$ is formed of: a flop, a flip, and a divisorial contraction to $\Proj_{y_2:y_3}^1 \subset X'$.
\end{theorem}

\begin{proof}
	The point $P_{y_1} \in Z_2$ because the weights of $M$ are not as in \ref{case B}. We connect this proof to the one of Theorem \ref{d2 = d3 special case}.
\end{proof}

\subsection{Proof of \ref{1=2>3>4} and \ref{1=2>3=4}}

Now we study the case where $d_1 = d_2$. Both \ref{1=2>3>4} and \ref{1=2>3=4} share the same behaviour when crossing of the ray $\rho_{y_1,y_2}$.

\begin{theorem} \label{theorem d1 = d2}
	Suppose $d_1 = d_2$. Then, there are two analytic flips (\textit{simultaneous flips} as in Remark \ref{simultaneous flips}) based at two distinct points in $Z_2$. 
\end{theorem}
To fix ideas, $M$ is in Tom$_1$ format throughout this proof. Here the specialisations to \ref{1=2>3>4} and \ref{1=2>3=4} of \ref{case A}, \ref{case B} still apply, but this time there two different variables that fit the entries with weight $d_1 = d_2$ (which replace the weight $\pi$ in \ref{case A}, \ref{case B}). 

Geometrically, $\alpha_2$ contracts the locus $\mathbb{A}_2$ to a line $\Proj^1_{y_1 : y_2} \subset \mathbb{G}_2$. So, the intersection $\mathbb{A}_2 \cap Y_2$ is mapped to $\Proj^1_{y_1 : y_2} \cap Z_2$. In Lemma \ref{d1 = d2 case A lemma} and in Remark \ref{d1 = d2 case B lemma} we discuss the nature of the intersection $\Proj^1_{y_1 : y_2} \cap Z_2$ in cases \ref{case A} and \ref{case B} respectively. The idea is that $\Proj^1_{y_1 : y_2}$ cuts out a rank 2 quadratic form in $y_1, y_2$, which determines two points in $Z_2$. Therefore, the variety $Y_2$ is subjected to two simultaneous flips.

\begin{proposition} \label{quad form on Z2}
	There exists a rank 2 quadratic form in $y_1, y_2$ defined on $\mathbb{G}_2$ that determines two distinct points $P_1, P_2$ in $Z_2$.
\end{proposition}
\begin{proof}
	Independently on \ref{case A} and \ref{case B}, without loss of generality we assume that $y_1$ occupies the $a_{25}$ entry and that $y_2$ occupies the $a_{34}$ entry of $M$. 
	The equations of $Z_2$ are in terms of $t$ as well, being the image of $Y_2$ via $\alpha_2$. If any of $y_1$ or $y_2$ is in one of the entries in the top row of the matrix, it picks up a $t$ factor in the blow up of $X$, so it vanishes when restricted to $\Proj^1_{y_1 : y_2}$. Moreover, if $y_1$ and $y_2$ appear in other entries of $M$ they need to be multiplied by some other variable. 
	Therefore, the quadratic form is to be found in the first pfaffian of $M$, i.e. it is the restriction of $\Pf_1(M)$ to $\Proj^1_{y_1 : y_2}$. It is of the form $y_1^2 - y_1 y_2 + y_2^2$ in case \ref{case A}, whereas it is $y_1^2 - y_1 y_2$ in case \ref{case B}. No other monomials, also in other equations, survive the restriction to $\Proj^1_{y_1 : y_2}$. For \ref{case A}, \ref{case B} the two quadratic forms describe two distinct points on $Z_2$.
\end{proof}

\begin{lemma} \label{d1 = d2 case A lemma}
	Let $Z_1$ be defined by a graded matrix $M$ in Tom format having weights as in \ref{case A}. Then, the two flips have exactly the same weights. 
\end{lemma}

\begin{proof}[Proof of Lemma \ref{d1 = d2 case A lemma}]
	Let $M$ have weights as in \ref{case A}. Place $y_1$ and $y_2$ in the entries $a_{25}$ and $a_{34}$ respectively. Locally at $P_{y_1}$ we can eliminate a potential linear term in the entries $a_{12}$ and $a_{15}$; likewise, locally at $P_{y_2}$ for a linear term in $a_{13}$ and $a_{14}$. Since $a_{12}$ and $a_{13}$ have the same weights, $y_1$ and $y_2$ eliminate the same variable when localising at their respective coordinate points; or otherwise they do not eliminate any variable in those entries at all (same for $a_{14}$ and $a_{15}$). The variables $y_3$ and $y_4$ cannot be eliminated, as they are always multiplied by a $t$ factor on the top row. Therefore, the birational transformations at $P_1$ and $P_2$ can only be flips. This proves that $\alpha_2$ contracts two loci of the same dimension: in fact, those loci are isomorphic. In conclusion, the flip phenomenon is symmetrical over $P_1, P_2 \in Z_2$.
\end{proof}

\begin{remark} \label{d1 = d2 case B lemma}
	The Lemma above does not hold if $M$ is as in \ref{case B}. In fact, if one of the flips is toroidal/hypersurface, the other one is not necessarily toroidal/hypersurface. This is because the weights in the top row of $M$ are all different, so $y_1$ and $y_2$ cannot eliminate the same variables, and therefore the flips at $P_1$ and $P_2$ cannot have the same weights. Moreover, suppose that a certain linear variable $w$ occupies the entry $a_{14}$: it can appear in the $a_{15}$ entry only if multiplied by a polynomial $f_{d_1-\upsilon}$ of degree $d_1-\upsilon$. Thus, there is no hope for $y_2$ to eliminate $w$, and therefore the two flips can have different numbers of weights.
\end{remark}

\begin{proof}[Proof of Theorem \ref{theorem d1 = d2}]
	Similarly to Theorem \ref{it is a flip}, $\Psi_2$ is an algebraically irreducible flip. Its restriction to $Y_2$ is constituted of two distinct components, each contracted to one of the points $P_1,P_2 \in Z_2$ (by Proposition \ref{quad form on Z2}). Lemma \ref{d1 = d2 case A lemma} and Remark \ref{d1 = d2 case B lemma} clarify the nature of such components.
\end{proof}

Theorem \ref{theorem d1 = d2} holds if $d_1 = d_2 > d_3 = d_4$ and $d_1 = d_2 > d_3 > d_4$, although the continuation of the link is different in the two cases. For the latter, the statements made for item \ref{1>2>3>4} still hold. For the former, we have that

\begin{theorem} \label{thm dP fibr}
	If $d_2>d_3=d_4$, then $\Phi'$ is a del Pezzo fibration over $\Proj^1_{y_3,y_4}$.
\end{theorem}
\begin{proof}
	Consider 
	$\Phi' \colon \F_3 \rightarrow \Proj^1_{y_3,y_4}$. The grading of $\F_3$ can be written as
	\begin{equation*}
	\left(
	\begin{array}{c c c c c c c | c c}
	t & s & x_1 & x_2 & x_3 & y_1 & y_2 & y_3 & y_4 \\
	d_3 & r+d_3 & a & b & c & d_2-d_3 & d_2-d_3 & 0 & 0 \\
	-1 & -1 & 0 & 0 & 0 & 1 & 1 & 1 & 1
	\end{array}
	\right) \; .
	\end{equation*}
	This is a weighted $\Proj^6$-bundle over $\Proj^1$. The intersection of $Y_3$ with its general fibre has dimension 2, for $y_3$ and $y_4$ now act as parameters. The restriction of $K_{Y_3}$ to such intersection is still ample. Hence, $\Phi'$ is a del Pezzo fibration of $Y_3$ over $\Proj^1_{y_3,y_4}$.
\end{proof}
\begin{lemma} \label{general fibre of dP is smooth}
	The intersection of $Y_3$ with the general fibre of the bundle defined by $\Phi'$ is smooth.
\end{lemma}
\begin{proof}
	The generic fibre $S$ of $\Phi'$ is a surface in $Y_3$. Suppose $S$ is singular. In particular, its closure inside the 3-fold $Y_3$ is a line. Therefore, $Y_3$ would contain a whole singular line, which is a contradiction with $Y_3$ being terminal.
\end{proof}

In \cite{BigTableLinks} we compute the degree of the general fibre of these del Pezzo fibrations.

\subsection{Proof of \ref{1>2>3=4}}

Similarly to the cases above, the weights of $M$ influence the behaviour of the link, and the distinction of \ref{case A}, \ref{case B} plays a crucial role. The majority of Hilbert series that fall into case \ref{1>2>3=4} are such that the weights of $M$ are in configuration \ref{case B}.

\begin{proposition}
	Suppose $M$ has weights in configuration \ref{case B}. Then, either $y_1$ appears as a square in the equations of $Y_2$, or $y_2$ appears as a square in the equations of $Y_3$.
\end{proposition}

\begin{proof}
	If the weights of $M$ are in \ref{case B}, $\Pf_1(M)$ involves the multiplication of the entries $a_{25}, a_{34}$ of same weight (either $d_1$ or $d_2$, depending on the specific Hilbert series considered). In contrast to the proof of Proposition \ref{quad form on Z2}, by hypothesis here only one variable has weight $d_1, d_2$, i.e. $y_1$ and $y_2$ respectively. Therefore, the quadratic form defined on $\mathbb{G}_2$ (or $\mathbb{G}_3$ respectively) is $y_1^2$ (or $y_2^2$ in turn).
\end{proof}

\begin{lemma}
	If $M$ has weights in configuration \ref{case B}, then either $\Psi_2$ or $\Psi_3$ is an isomorphism when restricted to $Y_2$ and $Y_3$ respectively.
\end{lemma}

\begin{proof}
	Either $y_1^2$ appears in the equations of $Y_2$, or $y_2^2$ appears in the equations of $Y_3$.
	We conclude the proof using Proposition \ref{skipped flip in >>>}.
\end{proof}

\begin{remark}
	In case \ref{1>2>3=4}, only the Hilbert series \#20544 has a weight configuration of type \ref{case A}. Since the only variable with weight $d_2$ is $y_2$, it is possible to cancel out $y_2$ from the entries $a_{25}$ and $a_{34}$ via row/column operations. Therefore the equations of $X$ have the monomial $y_2^2$. Nonetheless, no flip is missed because, performing the blow-up of $X$ and then saturating over $t$, the term $y_2^2$ picks up a $t$ factor.
	
	The weights of $M$ relative to the Hilbert series \#5516, \#5867, \#11437 are neither in configuration \ref{case A} nor \ref{case B}. Thus, $\Psi_2, \Psi_3$ are flips for $Y_2, Y_3$ respectively.
\end{remark}

The last map $\Phi'$ of the link in \ref{1>2>3=4} is a del Pezzo fibration, as in Theorem \ref{thm dP fibr}.

\subsection{Proof of \ref{1>2=3=4}}

There are six Hilbert series falling in case $d_1>d_2=d_3=d_4$. 

\begin{proposition}
	The birational link starting from the Hilbert series \#6865 is such that the restriction to $Y_2$ of the birational map $\Psi_2$ is an isomorphism.
\end{proposition}
\begin{proof}
	The weights of $M$ are as in \ref{case B}, so $y_1^2$ appears in the equations of $Y_2$.
\end{proof}

The other five Hilbert series falling in this case behave as expected.
\begin{proposition}
	Consider the birational link starting from $X$ as in one of the five Hilbert series left in case \ref{1>2=3=4}. Then, the restriction to the variety $Y_2$ of $\Psi_2$ is a flip for $Y_2$. 
\end{proposition}
\begin{proof}
	The weights of $M$ are neither in case \ref{case A} nor \ref{case B}. Thus, none of the ideal variables appears as a pure power in the equations of $Y_2$. 
\end{proof}

The end of the link is a conic bundle over a plane $\Proj^2$ with coordinates $y_2, y_3, y_4$.

\begin{proposition} \label{last map is conic bundle}
	The map $\Phi'$ is a conic bundle over the projective plane $\Proj^2_{y_2, y_3, y_4}$.
\end{proposition}
\begin{proof}
	Localise $\F_3$ at the plane $\Proj^2(d_2,d_2,d_2)_{y_2,y_3,y_4}$. Eliminate $s$ globally and discard the unprojection equations. We exclude $s$ from the grading of $\F_3$.
	\begin{equation*}
	\left(
	\begin{array}{c c c c c | c c c}
	t & x_1 & x_2 & x_3 & y_1 & y_2 & y_3 & y_4 \\
	d_2 & a & b & c & d_1-d_2 & 0 & 0 & 0 \\
	-1 & 0 & 0 & 0 & 1 & 1 & 1 & 1
	\end{array}
	\right) \; ,
	\end{equation*}
	so $\F_3$ is a weighted $\Proj^4$-bundle over $\Proj^2$. Above each point of $\Proj^2(d_2,d_2,d_2)_{y_2,y_3,y_4}$ we can locally eliminate two variables among $t, x_1, x_2, x_3, y_1$ via two of the pfaffian equations. The remaining three equations lie in the same principal ideal generated by one of them, which is a conic in the three surviving variables of the fibre with coefficients in the base $\Proj^2_{y_2, y_3, y_4}$.
\end{proof}

\subsection{Proof of \ref{1=2=3>4}}
In this case there are no flips occurring, and the links evolve as follows: $\psi_1$ is $n$ simultaneous flops by Theorem \ref{first step is a flop}, followed by a divisorial contraction $\Phi'$ to a Fano 3-fold $X'$ (by Lemma \ref{last map is div contr} and because  $d_4-d_1 < 0$).

\subsection{Proof of \ref{1=2=3=4}}
Here the first $n$ flops are followed by a conic bundle over the base $\Proj^3(d_1,d_1,d_1,d_1)_{y_1,y_2,y_3,y_4}$, and a similar statement to Proposition \ref{last map is conic bundle} holds.

\section{Examples} \label{examples}

In this section we present some explicit examples, highlighting the main phenomena described in Theorem \ref{T outputs of sarkisov links}. Recall that all the Fano 3-folds $X$ in this paper, and in particular in the examples of this section, can be explicitly constructed by means of Type I unprojections and are $\Q$-factorial following \cite{T&Jpart1}.

\subsection{Example for \ref{1>2>3>4}: \#10985, Tom$_1$} \label{10985 T1}

Consider $X \ni p$ where $X$ is the Tom type Fano 3-fold associated to the Hilbert series \#10985 and $p \in X$ is the Tom centre $\tfrac{1}{2}(1,1,1)$ in the basket of singularities $\mathcal{B}_X = \{ \tfrac{1}{2}(1,1,1), \tfrac{1}{6}(1,1,5)\}$. The ambient space of $X$ is $\Proj^7(1^3,2,3,4,5,6)$, with coordinates $x_1,x_2,x_3$, $s$, $y_4,y_3,y_2,y_1$ respectively. The divisor $D\cong \Proj_{x_1,x_2,x_3}(1,1,1)$ is defined by the ideal $I_D=\Span[y_1,y_2,y_3,y_4]$ and $D \subset Z_1$ for $M$ in Tom$_1$. There are 24 nodes on $D\subset Z_1$ (cf \cite{TJBigTable}).
To summarise, we are looking at the following varieties:
\begin{equation*}
\begin{array}{c c c c c}
\# 10985 & X & \subset \Proj^7(1^3,2,3,4,5,6) & \text{codimension } 4 & \{\tfrac{1}{2}(1,1,1), \tfrac{1}{6}(1,1,5)\} \\
\# 10962 & Z_1 & \subset \Proj^6(1^3,3,4,5,6) & \text{codimension } 3 & 24 \text{ nodes on $D$}
\end{array}
\end{equation*}

We aim to put ideal variables in an ideal entry having their same weight, and do analogously for the orbinates. The rest of the entries can be occupied by general polynomials in the given degrees, accordingly to the Tom$_1$ constraints. In this specific case, we end up with the following explicit matrix as in Section \ref{unprojection setup}
\begin{equation}
M=
\left(
\begin{array}{c c c c}
x_1 & -x_2 x_3 & -x_2^3 + y_4 & -x_3^4 + y_3 \\
& y_4 & y_3 & y_2 \\
& & x_2^2 y_4 - y_2 & y_1 \\
& & & x_1^4 y_4
\end{array}
\right) \; .
\end{equation}
The unprojection algorithm produces nine equations defining $X$. 
The blow-up $Y_1$ of $X$ at the Tom centre $p=P_s$ is contained in the rank 2 toric variety $\F_1$ with grading as in Proposition \ref{shape of F tom}, whose ray-chamber structure is described in Figure \ref{Mori cone 10985}.
\begin{figure}[H]
	\begin{tikzpicture}
	\draw[thick] (0,0) -- (0,1);
	\node [above left] at (0,1) {$t$};
	
	\draw[thick] (0,0) -- (2,1);
	\node [above left] at (2,1) {$s$};
	
	\draw[thick] (0,0) -- (1,0);
	\node [above right] at (1,0) {$x_1,x_2,x_3$};
	
	\draw[thick] (0,0) -- (6,-1);
	\node [above right] at (6,-1) {$y_1$};
	
	\draw[thick] (0,0) -- (5,-1);
	\node [below right] at (5,-1) {$y_2$};
	
	\draw[thick] (0,0) -- (4,-1);
	\node [below right] at (4,-1) {$y_3$};
	
	\draw[thick] (0,0) -- (3,-1);
	\node [below right] at (3,-1) {$y_4$};
	\end{tikzpicture}
	\caption{Mori cone of $\F_1$ for \#10985, Tom$_1$} \label{Mori cone 10985}
\end{figure}
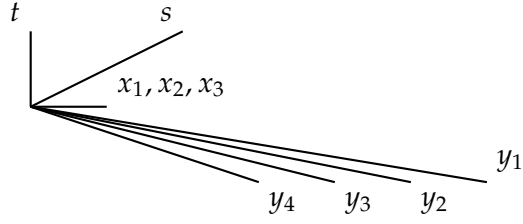

The Kawamata blow-up of the Tom centre $P_s$ is 
\begin{align}
\begin{split}
\Phi \colon \F_1 &\longrightarrow \Proj^7(1^3,2,3,4,5,6) \\
(t,s,x_1,x_2,x_3,y_1,y_2,y_3,y_4) &\longmapsto (x_1 t^{\frac{1}{2}}, x_2 t^{\frac{1}{2}}, x_3 t^{\frac{1}{2}}, y_4 t^{\frac{5}{2}}, y_3 t^{\frac{6}{2}}, y_2 t^{\frac{7}{2}}, y_1 t^{\frac{8}{2}}, s) \; .
\end{split} 
\end{align}

We record here only the pfaffian equations of $Y_1$. 
\begin{equation} \label{equations for Y >>>}
\begin{cases}	
t y_4^2 + x_1 x_2^2 y_4 - x_1 y_2 - x_2^3 y_4 + x_2 x_3 y_3=0 \\

-t y_4 y_3 - x_1 y_1 - x_2 x_3 y_2 + x_3^4 y_4 =0 \\

-t y_4 y_2 + t y_3^2 + x_1^5 y_4 + x_2^3 y_2 - t x_3^4 y_3 =0 \\

t y_4 y_1 + t y_3 y_2 + x_1^4 x_2 x_3 y_4 + x_1 x_2^2 y_1 + x_2^3 x_3 y_2 - x_2^3 y_1 - x_3^4 y_2 =0 \\

x_1^4 y_4^2 + x_2^2 y_4 y_2 - y_3 y_1 - y_2^2 =0
\end{cases}
\end{equation} 
From Theorem \ref{first step is a flop}, crossing the ray $\rho_{x_i}$ gives that $\Psi_1$ consists of 24 simultaneous flops based at the 24 nodes of $Z_1$. Since the weights of $M$ are in configuration \ref{case B}, then either $\psi_2$ or $\psi_3$ is an isomorphism by Proposition \ref{skipped flip in >>>}; $y_2$ appears as a pure power in \eqref{equations for Y >>>}, so $\psi_3$ is an isomorphism. To study $\psi_2$ we need to localise at $P_{y_1} \in Z_2$, so we look at the equations \ref{equations for Y >>>} locally analytically in a neighbourhood of the point $P_{y_1} \in Z_2$. Practically, $y_1$ is a local coordinate and we perform row operations on $\F_2$ in order to write the weight of $y_1$ as either $(\substack{\scriptscriptstyle{\pm1}\\\scriptscriptstyle{0}})$ or $(\substack{\scriptscriptstyle{0}\\\scriptscriptstyle{\pm1}})$. So, the grading of $\F_2$ becomes
\begin{equation*}
\left(
\begin{array}{c c c c c | c c c c}
t & s & x_1 & x_2 & x_3 & y_1 & y_2 & y_3 & y_4 \\
6 & 8 & 1 & 1 & 1 & 0 & -1 & -2 & -3 \\
1 & 1 & 0 & 0 & 0 & -1 & -1 & -1 & -1
\end{array}
\right) \; .
\end{equation*}
The flip $\Psi_2$ has weights $(6,8,1,1,1,-1,-2,-3)$; this stands for the contraction by $\alpha_2$ of $\Proj^4_{t,s,x_1,x_2,x_3}(6,8,1,1,1)$ to the point $P_{y_1} \in Z_2$, and the extraction by $\beta_2$ of $\Proj^2_{y_2,y_3,y_4}(1,2,3)$ from $P_{y_1}$. However, the intersection $\Proj^4_{t,s,x_2,x_3}(6,8,1,1,1) \cap Y_2$ can be a projective space smaller than $\Proj^4$. Analogously, this might hold for $\Proj^1_{y_2,y_4}(1,2,3)\cap Y_3$. The study of these intersections is done via the following argument.

Localising at the base of the isomorphism in codimension 1, $\Psi_i$, it is possible to write some of the variables as function of the others using the equations of $Y_i$. Examining the equations of $Y_2$ locally analytically at a neighbourhood of $P_{y_1} \in Z_2$ and considering $y_1$ as a local coordinate, we can set $y_1=1$ in the equations \eqref{equations for Y >>>}. Some linear monomials will emerge in the equations of $Y_2$ evaluated at $y_1=1$: those variables appearing linearly in $Y_2 {\big|}_{y_1=1}$ can be expressed in terms of the other variables locally analytically. Thus, we can locally eliminate them. In this specific case, the evaluation of \eqref{equations for Y >>>} at $y_1=1$ shows that $s, x_1, y_3$ appear linearly. Therefore, the weights of the flip for $Y_2$ are $(6,1,1,-1,-3)$, associated to the remaining variables $t, x_2, x_3, y_2,y_4$ respectively. Observe that it looks like that $\alpha_2$ contracts a 2-dimensional locus inside $Y_2$ to the point $P_{y_1}$, thus $\alpha_2$ does not seem like a small contraction, as required in flips. However, among the equations left after the local elimination process there is one involving $t$ and $y_4$: that is $\Pf_2=0$. This means that there is an equation cutting out the contracted locus by one dimension.

In conclusion, $\psi_2$ is a flip with weights $(6,1,1,-1,-3; 3)$, where the last 3 in this notation tracks the degree of the equation involving the monomial $t y_4$. In other words, a weighted projective space $\Proj_{t,x_2,x_3}(6,1,1)$ containing a hypersurface of degree 3 with coefficients in $\Proj_{y_2,y_4}(1,3)$ is flipped to $\Proj_{y_2,y_4}(1,3)$. In particular, a $\tfrac{1}{6}(1,1,5)$ singularity in $Y_2$ is contracted to $P_{y_1}$ via $\alpha_2$, and a $\tfrac{1}{3}(1,1,2)$ is extracted in $Y_3$ from $P_{y_1}$ via $\beta_2$. This is a hypersurface flip. Despite the fact that there are three surviving equations after the elimination process, the equation cutting out $\Proj_{t,x_2,x_3}(6,1,1)$ is only one: the other two are multiples of it, that is, $\Pf_2$ is the generator of the principal ideal of $Y_2$ on $\Proj_{t,x_2,x_3}(6,1,1)$. The map $\Psi_3$ based at $P_{y_2}$ defines a flip from $\F_3$ to $\F_4$, but its exceptional locus does not intersect $Y_3$, which is therefore not affected by $\Psi_3$. 
The last map of the link is $\Phi' \colon \F_4 \rightarrow \mathbb{G}_4$, defined by the linear system $\binom{4}{-1}$: it contracts the exceptional locus $\mathbb{E}' = \{ y_4 = 0\}$ to the point $P_{y_3} \in \mathbb{G}_4$. Explicitly, it is 
\begin{align}
\begin{split}
\Phi' \colon \F_4 &\longrightarrow \mathbb{G}_4 = \Proj^7(1,1,1,1,2,3,3,5) \\
(t,s,x_1,x_2,x_3,y_1,y_2,y_3,y_4) &\longmapsto (x_1 y_4, x_2 y_4, x_3 y_4, y_3, y_2 y_4, y_1 y_4^2, t y_4^3, s y_4^6) \; .
\end{split}
\end{align}

The exceptional locus $\mathbb{E}'$ is isomorphic to $\Proj^7(4,6,1,1,1,2,1)$ with coordinates $t,s,x_1,x_2,x_3,$ $y_1,y_2$ respectively: their weights are retrieved performing a localisation at $P_{y_3}$ as before. However, the intersection $\mathbb{E}' \cap Y_4$ is $\Proj^3(1,1,1,1)$, as we can eliminate the variables $t, s, y_1$ locally analytically in a neighbourhood of $P_{y_3}$.

We call $X'$ the push-forward $\Phi_*'(Y_4)$ of $Y_4$ via $\Phi'$. Practically, $y_4$ plays for $\Phi'$ the role that $t$ played for $\Phi$, being the extra variable needed to perform a blow-up: in this case, $\Phi'$ blows up the point $P_{y_3} \in X'$. The equations of $X'$ are therefore given by evaluating the equations of $Y_4$ at $y_4=1$. Observe that this shows that the variables $t$ and $s$ can be algebraically expressed as functions of the other variables: two equations of $Y_4{\big|}_{y_4 =1}$ are removed in order to perform this global elimination. 

Call $\varsigma_i$ for $i \in \{ 1, \dots, 8\}$ the coordinates of $\mathbb{G}_4$.
Since we globally eliminated $t,s$, then $X' \subset w\Proj' \subset \mathbb{G}_4$, where $w\Proj' := \Proj^5(1,1,1,1,2,3)$ with coordinates $\varsigma_1, \dots, \varsigma_6$. So, $\Phi'$ restricts to $\phi' \colon Y_4 \rightarrow X' \subset \Proj^5(1,1,1,1,2,3)$. The minimal basis of the ideal generated by the surviving equations of $Y_4{\big|}_{y_4 =1}$ give the explicit equations of $X'$, both of degree 4, are
\begin{equation}
\begin{cases}
\varsigma_1 \varsigma_2^2 \varsigma_4 - \varsigma_1 \varsigma_4 \varsigma_5 - \varsigma_1 \varsigma_6 - \varsigma_2^3 \varsigma_4 + \varsigma_2 \varsigma_3 \varsigma_4^2 - \varsigma_2 \varsigma_3 \varsigma_5 + \varsigma_3^4 = 0 \\
\varsigma_1^4 + \varsigma_2^2 \varsigma_5 - \varsigma_4 \varsigma_6 - \varsigma_5^2 =0
\end{cases} \; .
\end{equation}
In addition, it is possible to keep track of the singularities throughout the link. That is: $X$ has $\tfrac{1}{2}(1,1,1)$ and $\tfrac{1}{6}(1,1,5)$ singularities. After the blowup $\Phi$, $Y_1$ has only a singularity of type $\tfrac{1}{6}$: this holds for $Y_2$ too, as the basket does not change after the flops. The hypersurface flip $\Psi_2$ replaces $\tfrac{1}{6}(1,1,5)$ with $\tfrac{1}{3}(1,1,2)$, so $Y_3$ has one singularity of type $\tfrac{1}{3}$; same for $Y_4$, given that $Y_3$ and $Y_4$ are actually isomorphic. Lastly, $\phi'$ contracts a smooth locus, so the $\tfrac{1}{3}$ singularity of $Y_4$ is maintained in $X'$.

Now that we know the equations of $X'$ and their degrees, the basket of $X$ and its ambient space we deduce that $X'$ is a representative of the family \#16204 in \cite{grdb}, which is a Fano 3-fold in codimension 2.

\subsection{Example for \ref{1=2>3=4}: \#20652, Tom$_1$} \label{20652 T1}

Consider $X \ni p$ where $X$ is the Tom type Fano 3-fold associated to \#20652 and $p \in X$ is the Tom centre $\tfrac{1}{2}(1,1,1)$. Its ambient space is $\Proj^7(1^5,2^3)$, with coordinates $y_1,y_2,x_1$, $x_2$, $x_3,y_3,y_4,s$ respectively, $D\cong \Proj_{x_1,x_2,x_3}(1,1,1)\subset Z_1$ has 7 nodes, and $M$ is in Tom$_1$ format with entries 
\begin{equation}
M=
\left(
\begin{array}{c c c c}
x_1 & x_2 & x_3 & y_3 \\
& y_1 & y_2 & x_2 y_4 - x_3 y_3 + y_1 \\
& & x_1 y_3 - y_2 & y_4^2 - y_2 \\
& & & x_1 y_3 + y_1
\end{array}
\right) \; .
\end{equation}
\begin{equation*}
\begin{array}{c c c c c}
\# 20652 & X & \subset \Proj^7(1^5,2^3) & \text{codimension } 4 & \mathcal{B}_X = \{3 \times \tfrac{1}{2}(1,1,1)\} \\
\# 20543 & Z_1 & \subset \Proj^6(1^5,2^2) & \text{codimension } 3 & 7 \text{ nodes on $D$}
\end{array}
\end{equation*}
This time, the Mori cone of $\F_1$ is given by the fan in Figure \ref{Mori cone 20652}.
\begin{figure}[H]
	\begin{tikzpicture}
	\draw[thick] (0,0) -- (0,1);
	\node [above left] at (0,1) {$t$};
	
	\draw[thick] (0,0) -- (2,1);
	\node [above left] at (2,1) {$s$};
	
	\draw[thick] (0,0) -- (1,0);
	\node [above right] at (1,0) {$x_1,x_2,x_3$};
	
	\draw[thick] (0,0) -- (2,-1);
	\node [above right] at (2,-1) {$y_1,y_2$};
	
	\draw[thick] (0,0) -- (1,-1);
	\node [below right] at (1,-1) {$y_3,y_4$};
	\end{tikzpicture}
	\caption{Mori cone of $\F_1$ for \#20652, Tom$_1$} \label{Mori cone 20652}
\end{figure}
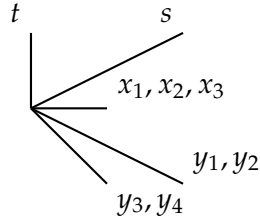

By Theorem \ref{first step is a flop}, $\Psi_1$ is given by 7 simultaneous flops.
The weights of $M$ are in configuration \ref{case A}, so there is a quadratic form determining two points $P_1, P_2$ in the intersection $Z_2 \cap \Proj^1_{y_1,y_2}$ (Proposition \ref{quad form on Z2}). Thus, Lemma \ref{d1 = d2 case A lemma} shows that the pencil of flips along the line $\Proj^1_{y_1,y_2} \subset \mathbb{G}_2$ restricts to two flips with base $P_1$ and $P_2$ respectively. So we look locally analytically in a neighbourhood of $P_1, P_2 \in Z_2$. Carrying out a similar calculation to the previous examples, we localise at $\Proj^1_{y_1,y_2} Z_2$. The weights of the flip of toric varieties based at $\Proj^1_{y_1,y_2}$ are $\left( 2,4,1,1,1,-1,-1 \right)$, where $\alpha_2$ contracts $\Proj^4_{t,s,x_1,x_2,x_3}(2,4,1,1,1)$ to $\Proj^1_{y_1,y_2}$, and $\beta_2$ extracts $\Proj^1_{y_3,y_4}$. We study the intersections $\Proj^4_{t,s,x_1,x_2,x_3}(2,4,1,1,1) \cap Y_2$ and $\Proj^1_{y_3,y_4} \cap Y_2$ locally analytically at a neighbourhood of $P_1$ and $P_2$ respectively. The first and second unprojection equations allow one to globally eliminate $s$ at $P_1$ and $P_2$. 
Similarly happens for $x_1$ using $\Pf_3(\alpha_1^*(M))$. On the other hand, we can use either $\Pf_4(\alpha_1^*(M))$ to eliminate $x_2$ at $P_1$, or $\Pf_2(\alpha_1^*(M))$ to eliminate $x_3$ at $P_2$. The intersection $\Proj^4_{t,s,x_1,x_2,x_3}(2,4,1,1,1) \cap Y_2$ is formed by two disjoint loci, generated by $t,x_2$ and $t,x_3$ at $P_1$ and $P_2$ respectively. Nonetheless, they determine two projective lines $\Proj^1(2,1)$. The fact that this elimination process has not excluded $y_3$ nor $y_4$ shows that $\Proj^1_{y_3,y_4} \subset Y_2$. The variable $t$ does not get eliminated because in $\Pf_1(\alpha_1^*(M))$ the polynomial $t \left( y_1^2 - y_1 y_2 + y_2^2 \right)$ appears: the variable $t$ could be eliminated only if $y_1^2 - y_1 y_2 + y_2^2 \not= 0$, but $P_1$ and $P_2$ are exactly the two solutions of $y_1^2 - y_1 y_2 + y_2^2 = 0$.

In conclusion, $\Psi_2$ restricts to two simultaneous Francia flips $( 2,1,-1,-1 )$ based at $P_1, P_2 \in Z_2$, as anticipated in Remark \ref{simultaneous flips}. In particular, two cyclic quotient singularities of type $\tfrac{1}{2}(1,1,1)$ in $Y_2$ are contracted to $P_1$ and $P_2$ respectively via $\alpha_2$, and $\beta_2$ extracts a smooth locus in $Y_3$. Therefore, $Y_3$ has Picard rank 2.

The last map of the link is the fibration $\Phi' \colon \F_4 \rightarrow \Proj^1_{y_3,y_4}$. Recall that  $-K_{Y_3} \sim \mathcal{O} \big( \substack{\scriptscriptstyle{1}\\\scriptscriptstyle{0}} \big)$. If $F$ is a general fibre of $\Phi'$, then $K_F = \left( K_{Y_3} + F \right) \big|_F = K_{Y_3} \big|_F$  by adjunction. Thus, $K_F$ is ample, $F$ a del Pezzo and, as a consequence, $\Phi'$ a del Pezzo fibration. The unprojection variable $s$ can be globally eliminated over each general point of $\Proj^1_{y_3,y_4}$. There is no other elimination that can be made. Therefore, the fibre $F$ of the del Pezzo fibration sits inside a projective space $\Proj^6$ with coordinates $t,x_1,x_2,x_3,y_1,y_2$. The matrix $M$ has become a matrix of linear forms in these variables. The equations of $F$ are the five (quadratic) maximal pfaffians of $M$. Therefore, the degree of $F$, and of the del Pezzo fibration, is 5.

\subsection{Example for \ref{1>2=3=4}: \#24097, Tom$_1$} \label{24097 T1}

Consider $X \ni p$ where $X \subset \Proj^7(1^6,2^2)$ is the Tom type Fano 3-fold \#24097, and $p \in X$ is the Tom centre $\tfrac{1}{2}(1,1,1)$. The coordinates of $\Proj^7(1^6,2^2)$ are $x_1,x_2,x_3$, $y_2, y_3,y_4,y_1,s$. The unprojection of $D\cong \Proj_{x_1,x_2,x_3}(1,1,1) \subset Z_1$ in Tom$_1$ format produces $X$, and there are 8 nodes on $D$. Here $Z_1$ is \#24077 defined by $M$: 
\begin{equation*}
M=
\left(
\begin{array}{c c c c}
x_1 & x_2 & x_3 & -y_2^2 - x_3 y_3 \\
& y_2 & y_3 & y_1 \\
& & y_4 & x_1 y_3 - y_4^2 \\
& & & -x_2 y_4 - x_3 y_4 + y_1
\end{array}
\right) \; .
\end{equation*}

After the 8 simultaneous flops of $\Psi_1$, the map $\Psi_2$ is a Francia flip $(2,1,-1,-1)$, and $\Phi' $ is a weighted $\Proj^5$-bundle over the projective space $\Proj^2_{y_2,y_3,y_4}(1,1,1)$. We show that $Y_3$ is a conic bundle over that base, and we compute its discriminant $\Delta$. Note that $Y_3$ is smooth. We record here the five pfaffian equations of $Y_3$.
\begin{equation*}
\begin{cases}
x_1 y_3^2 + x_2 y_2 y_4 + x_2 y_3 y_4 - x_1 y_4^2 - t y_3 y_4^2 - y_2 y_1 - y_4 y_1 =0 \\

x_1 x_3 y_3 + x_2^2 y_4 + x_2 x_3 y_4 + t^2 y_2^2 y_4 + t x_3 y_3 y_4 - t x_3 y_4^2 - x_2 y_1 =0 \\ 

t^2 y_2^2 y_3 + t x_3 y_3^2 + x_1 x_2 y_4 + x_1 x_3 y_4 - x_1 y_1 + x_3 y_1 =0 \\

t^2 y_2^3 - x_1^2 y_3 + t x_2 y_3^2 - t x_1 y_3 y_4 + t x_1 y_4^2 + x_2 y_1 =0 \\

x_3 y_2 - x_2 y_3 + x_1 y_4 =0 \\
\end{cases}
\end{equation*}

At a general point in $\Proj^2_{y_2,y_3,y_4}(1,1,1)$, it is possible to globally eliminate the variable $s$ thanks to the unprojection equations.

Now consider the line $\{ y_4 =0 \}$ in the base $\Proj^2_{y_2,y_3,y_4}(1,1,1)$, and look at its two affine patches $\{ y_2 \not=0 \}$ and $\{ y_3 \not=0 \}$. We want to study the conic equations above each of these patches: in fact, they both contribute to the discriminant $\Delta$.

Over the patch $\{ y_2 \not=0 \}$, $\Pf_5$ and $\Pf_1$ globally eliminate the variables $x_3$ and $y_1$ respectively: hence they are $x_3 = x_2 y_3$ and $y_1 = x_1 y_3^2$. Replace their expressions in the remaining three pfaffian equations, obtaining
\begin{equation*}
\begin{cases}
t^2 y_3 + t x_2 y_3^3 - x_1^2 y_3^2 + x_2 x_1 y_3^3 =0 \\  

x_1 x_2 y_3^2 - x_2 x_1 y_3^2 =0 \\  

t^2 - x_1^2 y_3 + t x_2 y_3^2 + x_2 x_1 y_3^2  =0 \\  
\end{cases} 
\end{equation*}
where $\Pf_2$ is identically zero, and $\Pf_3$ (above) is a multiple of $\Pf_4$ by a $y_3$ factor. Therefore, the conic that $\Pf_4$ describes is defined by the matrix 
\begin{equation*}
A_{y_2} = \left(
\begin{array}{c c c}
1 & 0 & \tfrac{1}{2}y_3^2 \\
0 & -y_3 & \tfrac{1}{2}y_3^2 \\
\tfrac{1}{2}y_3^2 & \tfrac{1}{2}y_3^2 & 0
\end{array}
\right)
\end{equation*}
as $(t,x_1,x_2) \cdot A_{y_2} \cdot (t,x_1,x_2)^T$.
Its determinant is $\det(A_{y_2}) = -\tfrac{1}{4}y_3^4 (1 + y_3)=0$. 

On the other hand, over the patch $\{ y_3 \not=0 \}$, $\Pf_1$ and $\Pf_5$ globally eliminate the variables $x_1$ and $x_2$ respectively: hence they are $x_1 = y_2 y_1$ and $x_2 = x_3 y_2$. Replace their expressions in the remaining three pfaffian equations: similarly to the other patch, the equation of the conic is $t^2 y_2^2 + t x_3 - y_2 y_1^2 + x_3 y_1=0$ given by $\Pf_3$. It is defined by the matrix $A_{y_3}$
\begin{equation*}
A_{y_3} = \left(
\begin{array}{c c c}
y_2^2 & \tfrac{1}{2} & 0 \\
\tfrac{1}{2} & 0 & \tfrac{1}{2} \\
0 & \tfrac{1}{2} & -y_2
\end{array}
\right)
\end{equation*}
determinant $\det(A_{y_3}) = -1/4 y_2 (1+y_2)$ and by the equation $(t,x_3,y_1) \cdot A_{y_3} \cdot (t,x_3,y_1)^T=0$.
Even though the contribution of $\det(A_{y_2})$ and $\det(A_{y_3})$ to the discriminant might look like $5+2=7$, the solutions to $\det(A_{y_2})=0$ and $\det(A_{y_3})=0$ overlap at the point $(-1,-1,0)$ which is counted twice. Therefore, $\Delta = 5+7-1=6$.

\subsection{\textbf{Comparison with Takagi}} \label{comparison with Takagi}

In \cite{Takagi}, the author classifies all possible extremal contractions $\Phi'$ appearing in sequences of flops and flips on $\Q$-factorial terminal Fano 3-folds $Y$ of Picard rank $\rho_Y =2$. These are Sarkisov links from certain $\Q$-Fano 3-folds $X$ with Picard rank 1 enjoying some additional properties (cf. \cite["Main Assumption 0.1"]{Takagi}). In particular, these varieties must have a singularity of type $\tfrac{1}{2}(1,1,1)$, that is blown up to initiate the link. Six of the varieties falling in Takagi's assumption are in codimension 4 and have a Type I centre. In particular, three of them are of Tom-type, and follow the description of Theorem \ref{T outputs of sarkisov links}. They are: \#24097 Tom$_1$ (above in Subsection \ref{24097 T1}, number 4.4 in Takagi's paper) falling in case $d_1=d_2=d_3<d_4$, \#20652 Tom$_1$ (above in \ref{20652 T1}, number 5.4) in case $d_1=d_2<d_3=d_4$, and \#16645 Tom$_1$ (number 2.2) in case $d_1<d_2=d_3=d_4$.

We have examined them here with our method, and we have showed that the outcomes predicted by Theorem \ref{T outputs of sarkisov links} match his results. The remaining three Hilbert series indicated by Takagi are of Jerry type. We omit their study from this paper.

\section{The Picard rank of certain codimension 4 terminal Fano 3-folds}
\label{section on Pic}

The Picard rank of quasi-smooth terminal Fano 3-folds in codimension 4 is unknown, except for some computational results in \cite{BrownFatighenti}. The construction carried out so far in this paper provides a tool to compute $\rho_X$ for certain Families of Fano 3-folds $X$ of Tom type in codimension 4. 

Theorem \ref{T outputs of sarkisov links} produces a birational link from each Fano 3-fold of Tom type.
Among 161 families of such Fano 3-folds, 96 have a link to another Fano 3-fold $X'$, and moreover $X$ and $X'$ have the same Picard rank. In 12 of these cases, $X'$ is quasi-smooth and so we may compute the rank directly. 
\begin{theorem} \label{Picard rank is 1}
	Let $X$ be a general Fano 3-fold of first Tom type and $p \in X$ its Tom centre. Suppose that the birational link for $X\ni p$ terminates with a quasi-smooth Mori fibre space $X' \rightarrow S$ with $\dim S =0$, that is, $X'$ is a Fano 3-fold. Then, the Picard rank of $X$ is $\rho_X = 1$. 
\end{theorem}
\begin{proof}
	Recall that $\Phi' \colon Y \rightarrow X'$ is an extremal divisorial contraction and $Y = Y_3, Y_4$ is a $\Q$-factorial Fano 3-fold. Hence, $X'$ is $\Q$-factorial.
	
	The Fano 3-fold $X'$ is quasi-smooth if the birational link for $X\ni p$ only involves toric flips and terminates with a divisorial contraction $\Phi'$ contracting the singular locus $\mathbb{E}'$ to a quasi-smooth point $p' \in X'$. 
	Since $\codim (X') \leq 3$ by Lemma~\ref{X' has smaller codim} and $X'$ is quasi-smooth, we apply \cite[Proposition 2.3]{PizzatoSanoTasin}, \cite[Tables 1, 2, 3]{BrownFatighenti} to conclude that $\rho_{X'}=1$. The birational link extracts and contracts exactly one irreducible divisor, and is otherwise an isomorphism in codimension 1. Therefore $\rho_X = \rho_{X'}=1$.
\end{proof}
In particular, in the hypotheses of Theorem \ref{Picard rank is 1} the link is a Sarkisov link. We expect that the hypotheses of quasi-smoothness of $X'$ can be lifted, and that Theorem \ref{Picard rank is 1} can be generalised to the rest of the 96 Tom families.
The Fano 3-fold in codimension 4 having Picard rank 1 are summarised in Table \ref{table of pic 1} together with their formats, Euler characteristic $e(X)$, and Hodge number $h^{2,1}(X)$, calculated using \cite[Theorem 4 and Table 3]{BrownFatighenti}. Note that \#11125 has two different birational links with ending with a quasi-smooth $X'$, as reported in Table \ref{table of pic 1}. Moreover, some of the Fano 3-folds in Table \ref{table of pic 1} admit other links to non-quasi-smooth Fano 3-folds, which therefore have Picard rank 1. This constitutes a further evidence that Theorem \ref{Picard rank is 1} could still hold for non-quasi-smooth $X'$.

\begin{table}[htb] 	
	\begin{tabular}{| c | c | c | c | c | c |}
		\hline
		GRDB ID & Embedding & Format & $e(X)$ & $h^{2,1}(X)$ & $\rho_X$ \\
		\hline
		\hline
		1169 & $\Proj^7(1,2,3,4,5,7^2,9)$ & T$_1$ & -38 & 21 & 1 \\
		\hline
		1253 & $\Proj^7(1,2,3,4^2,5^2,7)$ & $T_1$ & -24 & 14 & 1 \\
		\hline
		4925 & $\Proj^7(1^2,3,4,5,6,7^2)$ & T$_1$ & -56 & 30 & 1 \\
		\hline
		5177 & $\Proj^7(1^2,2,3,4,5^2,6)$ & T$_1$ & -48 & 26 & 1 \\
		\hline
		5279 & $\Proj^7(1^2,2,3^2,4,5^2)$ & T$_1$, $\frac{1}{5}$ & -38 & 21 & 1 \\
		\hline
		5305 & $\Proj^7(1^2,2,3^2,4^2,5)$ & T$_1$, $\frac{1}{5}$ &-36 & 20 & 1 \\
		\hline
		5963 & $\Proj^7(1^2,2^2,3^3,5)$ & T$_1$, $\frac{1}{3}$ &-28 & 16 & 1 \\
		\hline
		11005 & $\Proj^7(1^3,2,3^2,4,5)$ & T$_1$ & -46 & 25 & 1 \\
		\hline
		11125 & $\Proj^7(1^3,2^2,3^2,4)$ & T$_1$, $\frac{1}{2}$ & -32 & 18 & 1 \\
		\hline
		11125 & $\Proj^7(1^3,2^2,3^2,4)$ & T$_2$, $\frac{1}{4}$ & -32 & 18 & 1 \\
		\hline
		11455 & $\Proj^7(1^3,2^3,3^2)$ & T$_1$, $\frac{1}{3}$ & -22 & 13 & 1 \\
		\hline
		16339 & $\Proj^7(1^4,2^3,3)$ & T$_1$, $\frac{1}{2}$ & -22 & 13 & 1 \\
		\hline
	\end{tabular}
	\caption{Fano 3-folds in codimension 4 with Picard rank 1}\label{table of pic 1}
\end{table}

\section{Appendix} \label{appendix}

\subsection{Papadakis' algorithm for unprojection} \label{Papadakis' algorithm}

In \cite{PapadakisComplexes} the author explicitly builds the Type I unprojection equations from a codimension 3 Fano 3-fold $Z$ in Tom format. Here we briefly retrace the steps of Papadakis' construction, combining the two notations. Suppose for simplicity that the matrix $M$ is in format Tom$_1$. For $D \cong \Proj_{x_1,x_2,x_3}(a,b,c)$ the divisor in $Z$, and $I_D :=\Span[y_1,y_2,y_3,y_4]$, the graded matrix $M$ is 
\begin{equation} \label{general form T1} M =
\left(
\begin{array}{c c c c}
p_1 & p_2 & p_3 & p_4 \\
& a_{23} & a_{24} & a_{25} \\
& & a_{34} & a_{35} \\
& & & a_{45}
\end{array}
\right) \; .
\end{equation} 
Here the $a_{ij}$ are polynomials of the form $ a_{ij} := \sum_{k=1}^{4} \alpha_{ij}^k y_k $ for some polynomial coefficients~$\alpha_{ij}^k$. The $a_{ij}$ are in the ideal $I_D$. Instead, $p_j$ are not in $I_D$, in accordance to Definition \ref{Tom definition}. Only in this Appendix, we calculate $\Pf_i$ by excluding the $(i+1)$-th row and the $(i+1)$-th column for $i \in \{0,1,2,3,4\}$. Only $\Pf_1, \dots, \Pf_4$ are linear in the $y_i$; hence, there exists a unique matrix $Q$ such that $ \left( \Pf_1(M), \dots, \Pf_4(M) \right)^T = Q \left( y_1, \dots, y_4 \right)^T $. Explicitly, $ Q = \left( \Pf_i(N_j) \right)_{i,j=1 \dots 4}$
where 
\begin{equation*} 
N_i =
\left(
\begin{array}{c c c c}
p_1 & p_2 & p_3 &  p_4 \\
& \alpha^i_{23} & \alpha^i_{24} & \alpha^i_{25} \\
& & \alpha^i_{34} & \alpha^i_{35} \\
& & & \alpha^i_{45}
\end{array}
\right) \; .
\end{equation*}
and $\alpha^i_{kl}$ is the coefficient of $y_i$ in $a_{kl}$. Define $H_i$ as the vector of length 4 whose $i$-th entry is $(-1)^{i+1}$ times the determinant of the submatrix of $Q$ obtained by removing the $i$-th column and the $i$-th row. F or all $i,j \in \{ 1, \dots, 4\}$, the vectors $H_i$ satisfy
\begin{equation} \label{H_i independent on i}
p_i H_j = p_j H_i
\end{equation}
(cf Lemma 5.3 of \cite{PapadakisComplexes}). Thus, the quotient $H_i / p_i$ is independent of $i$. The polynomials $g_1, \dots, g_4$ are defined via the following equality of vectors of length 4 $(g_1, g_2,g_3,g_4) = H_i /p_i$.
For instance, $g_1$ is the determinant of the matrix obtained deleting the first column and the first row of $Q$ divided by $p_1$, i.e.
\begin{equation} \label{def of unproj in Pap}
g_1 = \frac{1}{p_1} \det \left(
\begin{array}{c c c}
\Pf_2(N_2) & \Pf_2(N_3) & \Pf_2(N_4) \\
\Pf_3(N_2) & \Pf_3(N_3) & \Pf_3(N_4) \\
\Pf_4(N_2) & \Pf_4(N_3) & \Pf_4(N_4)
\end{array}
\right) \; .
\end{equation}
The $g_j$ are the right hand sides of the unprojection equations, that is, the unprojection equations defining $X$ are $s y_j = g_j$ for $j=1, \dots, 4$.

\bibliography{Tom_Sarkisov_links_new}

\newcommand{\etalchar}[1]{$^{#1}$}
\begin{thebibliography}{BCDP21}

\bibitem[ABR02]{AltinokBrownReidK3}
S.~Alt{\i}nok, G.~Brown, and M.~Reid.
\newblock Fano 3-folds, {$K3$} surfaces and graded rings.
\newblock In {\em Topology and geometry: commemorating {SISTAG}}, volume 314 of
  {\em Contemp. Math.}, pages 25--53. Amer. Math. Soc., Providence, RI, 2002.

\bibitem[Ahm17]{HamidPliabilityCox}
H.~Ahmadinezhad.
\newblock On pliability of del {P}ezzo fibrations and {C}ox rings.
\newblock {\em J. Reine Angew. Math.}, 723:101--125, 2017.

\bibitem[AO18]{AhmadinezhadOkadaPfaff}
H.~Ahmadinezhad and T.~Okada.
\newblock Birationally rigid {P}faffian {F}ano 3-folds.
\newblock {\em Algebr. Geom.}, 5(2):160--199, 2018.

\bibitem[AZ16]{AhmadinezhadZucconiCI}
H.~Ahmadinezhad and F.~Zucconi.
\newblock Mori dream spaces and birational rigidity of {F}ano 3-folds.
\newblock {\em Adv. Math.}, 292:410--445, 2016.

\bibitem[AZ17]{AhmadinezhadZucconiCircles}
H.~Ahmadinezhad and F.~Zucconi.
\newblock Circle of {S}arkisov links on a {F}ano 3-fold.
\newblock {\em Proc. Edinb. Math. Soc. (2)}, 60(1):1--16, 2017.

\bibitem[BCDP21]{BlancCheltsovDuncanProkhorov}
J.~Blanc, I.~Cheltsov, A.~Duncan, and Y.~Prokhorov.
\newblock Birational self-maps of threefolds of (un)-bounded genus or gonality.
\newblock {\em arXiv preprint arXiv:1905.00940, to appear in Amer. J. Math.},
  2021.

\bibitem[BCZ04]{BCZ}
G.~Brown, A.~Corti, and F.~Zucconi.
\newblock Birational geometry of 3-fold {M}ori fibre spaces.
\newblock In {\em The {F}ano {C}onference}, pages 235--275. Univ. Torino,
  Turin, 2004.

\bibitem[BF20]{BrownFatighenti}
G.~Brown and E.~Fatighenti.
\newblock Hodge numbers and deformations of {F}ano 3-folds.
\newblock {\em Doc. Math.}, 5:267--307, 2020.

\bibitem[BK{\etalchar{+}}]{grdb}
G.~Brown, A.~M. Kasprzyk, et~al.
\newblock Graded {R}ing {D}atabase.
\newblock {\em Online. Access via \url{http://www. grdb. co. uk}}.

\bibitem[BKQ18]{brownP2xP2}
G.~Brown, A.~M. Kasprzyk, and M.~I. Qureshi.
\newblock Fano 3-folds in {$\Bbb P^2\times \Bbb P^2$} format, {T}om and
  {J}erry.
\newblock {\em Eur. J. Math.}, 4(1):51--72, 2018.

\bibitem[BKR12a]{T&Jpart1}
G.~Brown, M.~Kerber, and M.~Reid.
\newblock Fano 3-folds in codimension 4, {T}om and {J}erry. {P}art {I}.
\newblock {\em Compos. Math.}, 148(4):1171--1194, 2012.

\bibitem[BKR12b]{TJBigTable}
G.~Brown, M.~Kerber, and M.~Reid.
\newblock Tom and {J}erry table, part of "{F}ano 3-folds in codimension 4,
  {T}om and {T}erry. {P}art {I}".
\newblock {\em Compositio Mathematica}, 148(4):1171--1194, 2012.

\bibitem[BZ10]{BrownZucconi}
G.~Brown and F.~Zucconi.
\newblock Graded rings of rank 2 {S}arkisov links.
\newblock {\em Nagoya Math. J.}, 197:1--44, 2010.

\bibitem[Cam20]{BigTableLinks}
L.~Campo.
\newblock Big {T}able {L}inks.
\newblock {\em Online. Access via
  \url{http://www.grdb.co.uk/files/fanolinks/BigTableLinks.pdf}}, 2020.

\bibitem[CLO15]{CoxLittleOSheaIdVarAlgorithms}
D.~A. Cox, J.~Little, and D.~O'Shea.
\newblock {\em Ideals, varieties, and algorithms}.
\newblock Undergraduate Texts in Mathematics. Springer, fourth edition, 2015.
\newblock An introduction to computational algebraic geometry and commutative
  algebra.

\bibitem[CLS11]{CoxToricVarieties}
D.~A. Cox, J.~B. Little, and H.~K. Schenck.
\newblock {\em Toric varieties}, volume 124 of {\em Graduate Studies in
  Mathematics}.
\newblock American Mathematical Society, Providence, RI, 2011.

\bibitem[CM04]{CortiMella}
A.~Corti and M.~Mella.
\newblock Birational geometry of terminal quartic 3-folds. {I}.
\newblock {\em Amer. J. Math.}, 126(4):739--761, 2004.

\bibitem[Cor95]{Corti95}
A.~Corti.
\newblock Factoring birational maps of threefolds after {S}arkisov.
\newblock {\em J. Algebraic Geom.}, 4(2):223--254, 1995.

\bibitem[CP17]{CheltsovParkRigidHypersurfaces}
I.~Cheltsov and J.~Park.
\newblock Birationally rigid {F}ano threefold hypersurfaces.
\newblock {\em Mem. Amer. Math. Soc.}, 246(1167):v+117, 2017.

\bibitem[CPR00]{CPR}
A.~Corti, A.~Pukhlikov, and M.~Reid.
\newblock Fano {$3$}-fold hypersurfaces.
\newblock In {\em Explicit birational geometry of 3-folds}, volume 281 of {\em
  London Math. Soc. Lecture Note Ser.}, pages 175--258. Cambridge Univ. Press,
  Cambridge, 2000.

\bibitem[HM13]{HaconMcKernan}
C.~D. Hacon and J.~McKernan.
\newblock The {S}arkisov program.
\newblock {\em J. Algebraic Geom.}, 22(2):389--405, 2013.

\bibitem[IF00]{IanoFletcher}
A.~R. Iano-Fletcher.
\newblock Working with weighted complete intersections.
\newblock In {\em Explicit birational geometry of 3-folds}, volume 281 of {\em
  London Math. Soc. Lecture Note Ser.}, pages 101--173. Cambridge Univ. Press,
  Cambridge, 2000.

\bibitem[IP96]{IskovskikhPukhlikov}
V.~A. Iskovskikh and A.~V. Pukhlikov.
\newblock Birational automorphisms of multidimensional algebraic manifolds.
\newblock volume~82, pages 3528--3613. 1996.
\newblock Algebraic geometry, 1.

\bibitem[Kaw96]{kawamata1996divisorial}
Y.~Kawamata.
\newblock Divisorial contractions to {$3$}-dimensional terminal quotient
  singularities.
\newblock In {\em Higher-dimensional complex varieties ({T}rento, 1994)}, pages
  241--246. de Gruyter, Berlin, 1996.

\bibitem[Oka14]{OkadaCI}
T.~Okada.
\newblock Birational {M}ori fiber structures of {$\Bbb Q$}-{F}ano 3-fold
  weighted complete intersections.
\newblock {\em Proc. Lond. Math. Soc. (3)}, 109(6):1549--1600, 2014.

\bibitem[Pap04]{PapadakisComplexes}
S.~A. Papadakis.
\newblock Kustin-{M}iller unprojection with complexes.
\newblock {\em J. Algebraic Geom.}, 13(2):249--268, 2004.

\bibitem[PR04]{PapadakisReidKM}
S.~A. Papadakis and M.~Reid.
\newblock Kustin-{M}iller unprojection without complexes.
\newblock {\em J. Algebraic Geom.}, 13(3):563--577, 2004.

\bibitem[PST17]{PizzatoSanoTasin}
M.~Pizzato, T.~Sano, and L.~Tasin.
\newblock Effective nonvanishing for {F}ano weighted complete intersections.
\newblock {\em Algebra Number Theory}, 11(10):2369--2395, 2017.

\bibitem[Rei80]{ReidCanonical3folds}
M.~Reid.
\newblock Canonical {$3$}-folds.
\newblock In {\em Journ\'{e}es de {G}\'{e}ometrie {A}lg\'{e}brique d'{A}ngers,
  {J}uillet 1979/{A}lgebraic {G}eometry, {A}ngers, 1979}, pages 273--310.
  Sijthoff \& Noordhoff, Alphen aan den Rijn---Germantown, Md., 1980.

\bibitem[Tak02]{Takagi}
H.~Takagi.
\newblock On classification of {$\Bbb Q$}-{F}ano 3-folds of {G}orenstein index
  2. {I}, {II}.
\newblock {\em Nagoya Math. J.}, 167:117--155, 157--216, 2002.

\end{thebibliography}
\bibliographystyle{alpha}

\end{document}